\documentclass{amsart}
\usepackage{amsmath,amsfonts,amssymb,amsthm}
\usepackage{hyperref}

\usepackage{color}

\theoremstyle{plain}
\newtheorem{theorem}{Theorem}[section]
\newtheorem{proposition}{Proposition}[section]
\newtheorem{claim}{Claim}[section]

\newtheorem{corollary}{Corollary}[section]

\theoremstyle{definition}
\newtheorem{definition}{Definition}[section]

\newtheorem*{remark*}{Remark}
\newtheorem*{conjecture*}{Conjecture}
\newtheorem*{problem*}{Problem}

\newcommand{\A}{\mathcal{A}}
\newcommand{\B}{\mathcal{B}}
\newcommand{\C}{\mathcal{C}}
\newcommand{\D}{\mathcal{D}}

\newtheorem*{pullback}{Pullback Theorem}{\bfseries}{\itshape}
\newtheorem*{proposition*}{Proposition}{\bfseries}{\itshape}

\begin{document}
\title{Computable embeddings for pairs of linear orders}
\thanks{The first author was partially supported by NSF Grant DMS \#1600625, which allowed him to visit Sofia University. The first author was also supported by RFBR, project number 20-31-70006. 
  The second and third authors were partially supported by the Bulgarian National Science Fund through the project “Models of computability”, DN-02-16/19.12.2016.
  This paper is an extended and revised version of~\cite{BGV19CiE}.}

\author{Nikolay Bazhenov} 
\author{Hristo Ganchev} 
\author{Stefan Vatev}

\address{Sobolev Institute of Mathematics, Novosibirsk, Russia\ \and\ 
  Novosibirsk State University, Novosibirsk, Russia}
\email{bazhenov@math.nsc.ru}

\address{Department of Mathematical Logic\\
Sofia University\\
Bulgaria}
\email{ganchev@fmi.uni-sofia.bg}

\address{Department of Mathematical Logic\\
Sofia University\\
Bulgaria}
\email{stefanv@fmi.uni-sofia.bg}

\maketitle              

\begin{abstract}
  We study computable embeddings for pairs of structures, i.e. for classes containing precisely two non-isomorphic structures.
  Surprisingly, even for some pairs of simple linear orders, computable embeddings induce a non-trivial degree structure.
  Our main result shows that $\{\omega \cdot k,\omega^\star \cdot k\}$ is computably embeddable in $\{\omega \cdot t, \omega^\star \cdot t\}$ iff $k$ divides $t$.
\end{abstract}

\section{Introduction}

We study computability-theoretic complexity for classes of countable structures. A widely used approach to investigating algorithmic complexity involves comparing different classes of structures by using a particular notion of \emph{reduction} between classes. Examples of such reductions include computable embeddings \cite{CCKM04,CKM07}, Turing computable embeddings \cite{FKMQS11,KMV07}, $\Sigma$-reducibility \cite{EPS-11,Puz-09}, computable functors \cite{HMMM-17,MPSS-18}, enumerable functors \cite{Rossegger}, etc. If a class $\mathcal{K}_0$ is reducible to a class $\mathcal{K}_1$ and there is no reduction from $\mathcal{K}_1$ into $\mathcal{K}_0$, then one can say that $\mathcal{K}_1$ is computationally ``harder'' than $\mathcal{K}_0$. In a standard com\-pu\-ta\-bi\-li\-ty-theoretic way, a particular reduction gives rise to the corresponding degree structure on classes.
Nevertheless, note that there are other ways to compare com\-pu\-ta\-bi\-li\-ty-the\-o\-re\-tic complexity of two classes of structures, see, e.g., \cite{HKSS02,Miller}.

Friedman and Stanley~\cite{FS89} introduced the notion of \emph{Borel embedding} to compare complexity of the classification problems for classes of countable structures. Calvert, Cummins, Knight, and Miller \cite{CCKM04} (see also \cite{KMV07}) developed two notions, \emph{computable embeddings} and \emph{Turing computable embeddings}, as effective counterparts of Borel embeddings. The formal definitions of these embeddings are given in Section~\ref{sect:prelim}. Note that if there is a computable embedding from $\mathcal{K}_0$ into $\mathcal{K}_1$, then there is also a Turing computable embedding from $\mathcal{K}_0$ into $\mathcal{K}_1$. The converse is not true, see Section~\ref{sect:prelim} for details.

In this paper, we follow the approach of~\cite{CCKM04} and study computable embeddings for pairs of structures, i.e. for classes $\mathcal{K}$ containing precisely two non-isomorphic structures. Our motivation for investigating pairs of structures is two-fold. First, these pairs play an important role in computable structure theory. The technique of pairs of computable structures, which was developed by Ash and Knight \cite{AK90,AK00}, found many applications in studying various computability-theoretic properties of structures (in particular, their degree spectra and effective categoricity, see, e.g., \cite{AK00,Bazh-MS,GHKMMS}). 

Second, pairs of computable structures constitute the simplest case, which is significantly different from the case of one-element classes: It is not hard to show that for any computable structures $\mathcal{A}$ and $\mathcal{B}$, the one-element classes $\{ \mathcal{A}\}$ and $\{ \mathcal{B}\}$ are equivalent with respect to computable embeddings. On the other hand, our results will show that computable embeddings induce a non-trivial degree structure for two-element classes consisting of computable structures. 
The reader is referred to Section~\ref{subsect:background} for a further discussion of our motivation.

In this paper, we concentrate on the pair of linear orders $\omega$  and $\omega^{\star}$.
We denote by $\deg_{tc}(\{\omega, \omega^{\star}\})$ the degree of the class $\{ \omega, \omega^{\star}\}$ under Turing computable embeddings (to be defined in Section~\ref{sect:prelim}). Quite unexpectedly, it turned out that a seemingly simple problem of studying computable embeddings for classes from $\deg_{tc}(\{\omega, \omega^{\star}\})$ requires developing new techniques. 

The outline of the paper is as follows. Section~\ref{sect:prelim} contains the necessary preliminaries. In
Section~\ref{sect:tc-deg} we give a necessary and sufficient condition for a pair of structures $\{ \mathcal{A}, \mathcal{B}\}$ to
belong to $\deg_{tc}(\{ \omega, \omega^{\star}\})$. In Section~\ref{sect:top:pair} we show that the pair $\{1+\eta,\eta + 1\}$ is
the greatest element inside $\deg_{tc}(\{ \omega, \omega^{\star}\})$ with respect to computable embeddings.
In Section~\ref{sect:infinite:chain} we prove the following result: inside $\deg_{tc}(\{ \omega, \omega^{\star}\})$, there is an infinite chain of degrees induced by computable embeddings.
In the remaining of the paper, we use the techniques from Section~\ref{sect:infinite:chain} to give a characterization of the
degrees of pairs of the form $\{\omega \cdot k,\omega^\star \cdot k\}$ induced by computable embeddings.
We show that $\{\omega \cdot k,\omega^\star \cdot k\}$ is computably embeddable in $\{\omega \cdot t, \omega^\star \cdot t\}$ iff $k$ divides $t$.


\section{Preliminaries}\label{sect:prelim}

We consider only computable languages, and structures with domain contained in $\omega$. We assume that any considered class of structures $\mathcal{K}$ is closed under isomorphism, modulo the restriction on domains. In addition, we assume that all structures from $\mathcal{K}$ have the same language. For a structure $\mathcal{S}$, $D(\mathcal{S})$ denotes the atomic diagram of $\mathcal{S}$. We will often identify a structure and its atomic diagram.

Let $\mathcal{K}_0$ be a class of $L_0$-structures, and $\mathcal{K}_1$ be a class of $L_1$-structures. In the definition below, we use the following convention: An \emph{enumeration operator} $\Gamma$ is treated as a c.e. set of pairs $(\alpha,\varphi)$, where $\alpha$ is a finite set of atomic $(L_0\cup \omega)$-sen\-ten\-ces, and $\varphi$ is a atomic $(L_1\cup\omega)$-sentence.

\begin{definition}[{\cite{CCKM04,KMV07}}]
  An enumeration operator $\Gamma$ is a \emph{computable embedding} of $\mathcal{K}_0$ into $\mathcal{K}_1$, denoted by $\Gamma\colon \mathcal{K}_0 \leq_c \mathcal{K}_1$, if $\Gamma$ satisfies the following:
  \begin{enumerate}
  	\item For any $\mathcal{A}\in \mathcal{K}_0$, $\Gamma(\mathcal{A})$ is the atomic diagram of a structure from $\mathcal{K}_1$.

  	\item For any $\mathcal{A},\mathcal{B}\in \mathcal{K}_0$, we have $\mathcal{A}\cong\mathcal{B}$ if and only if $\Gamma(\mathcal{A}) \cong \Gamma(\mathcal{B})$.
  \end{enumerate}	
\end{definition}

Any computable embedding has the important property of \emph{monotonicity}: If $\Gamma\colon \mathcal{K}_0 \leq_c \mathcal{K}_1$ and $\mathcal{A}\subseteq \mathcal{B}$ are structures from $\mathcal{K}_0$, then we have $\Gamma(\mathcal{A}) \subseteq \Gamma(\mathcal{B})$ \cite[Proposition 1.1]{CCKM04}.

\begin{definition}[{\cite{CCKM04,KMV07}}]
  A Turing operator $\Phi=\varphi_e$ is a \emph{Turing computable embedding} of $\mathcal{K}_0$ into $\mathcal{K}_1$, denoted by $\Phi\colon \mathcal{K}_0 \leq_{tc} \mathcal{K}_1$, if $\Phi$ satisfies the following:
  \begin{enumerate}
  	\item For any $\mathcal{A}\in \mathcal{K}_0$, the function $\varphi^{D(\mathcal{A})}_e$ is the characteristic function of the atomic diagram of a structure from $\mathcal{K}_1$. This structure is denoted by $\Phi(\mathcal{A})$.

  	\item For any $\mathcal{A},\mathcal{B}\in \mathcal{K}_0$, we have $\mathcal{A}\cong\mathcal{B}$ if and only if $\Phi(\mathcal{A}) \cong \Phi(\mathcal{B})$.
  \end{enumerate}	
\end{definition}

\begin{proposition*}[Greenberg and, independently, Kalimullin; see \cite{Kal-18,KMV07}]
  If $\mathcal{K}_0 \leq_c \mathcal{K}_1$, then $\mathcal{K}_0 \leq_{tc} \mathcal{K}_1$. The converse is not true.
\end{proposition*}

Both relations $\leq_c$ and $\leq_{tc}$ are preorders. If $\mathcal{K}_0 \leq_{tc} \mathcal{K}_1$ and $\mathcal{K}_1\leq_{tc} \mathcal{K}_0$, then we say that $\mathcal{K}_0$ and $\mathcal{K}_1$ are \emph{$tc$-equivalent}, denoted by $\mathcal{K}_0 \equiv_{tc} \mathcal{K}_1$. For a class $\mathcal{K}$, by $\deg_{tc}(\mathcal{K})$ we denote the family of all classes which are $tc$-equivalent to $\mathcal{K}$. Similar notations can be introduced for the $c$-reducibility.

For $L$-structures $\mathcal{A}$ and $\mathcal{B}$, we say that $\mathcal{A}\equiv_1 \mathcal{B}$ if $\mathcal{A}$ and $\mathcal{B}$ satisfy the same $\exists$-sen\-ten\-ces. Let $\alpha$ be a computable ordinal. The formal definition of a \emph{computable $\Sigma_{\alpha}$ formula} (or a \emph{$\Sigma^c_{\alpha}$ formula}, for short) can be found in \cite[Chap. 7]{AK00}. By $\Sigma^c_{\alpha}$-$Th(\mathcal{A})$, we denote the set of all $\Sigma^c_{\alpha}$ sentences which are true in $\mathcal{A}$. 

Note that in this paper we will use $\Sigma^{c}_{\alpha}$ formulas only for $\alpha \leq 2$. Informally, the class of $\Sigma^c_2$ formulas can be described as follows. A \emph{$\Pi^c_1$ formula} is a c.e. conjunction of finitary $\forall$-formulas:
\[
	\xi(\bar y) = \bigwedge_{j\in W} \forall \bar z_j \theta_j (\bar y, \bar z_j),
\]
where the set $W$ is c.e. and every $\theta_j$ is a quantifier-free $L$-formula. A \emph{$\Sigma^c_2$ formula} is a c.e. disjunction of the form:
\[
	\psi(\bar x) = \bigvee_{i \in V} \exists \bar y_i \xi_i (\bar x, \bar y_i),
\]
where $V$ is c.e. and every $\xi_i$ is a $\Pi^c_1$ formula.

\begin{pullback}[Knight, Miller, and Vanden Boom~\cite{KMV07}]
  Suppose that $\Phi: \mathcal{K}_1\leq_{tc}\mathcal{K}_2$. Then for any computable infinitary sentence $\psi_2$ in the language of $\mathcal{K}_2$, one can effectively find a computable infinitary sentence $\psi_1$ in the language of $\mathcal{K}_1$ such that for all $\mathcal{A}\in \mathcal{K}_1$, we have $\mathcal{A} \models \psi_1$ if and only if $\Phi(\mathcal{A}) \models \psi_2$. Moreover, for a non-zero $\alpha <\omega^{CK}_1$, if $\psi_2$ is a $\Sigma^c_{\alpha}$ formula, then so is $\psi_1$.
\end{pullback}

When we work with pairs of structures, we use the following convention: Suppose that $\Gamma$ is a (Turing) computable embedding
from a class $\{ \mathcal{A}, \mathcal{B}\}$ into a class $\{ \mathcal{C}, \mathcal{D}\}$. Then for convenience, we always assume
that $\mathcal{A}$ and $\mathcal{B}$ are not isomorphic, $\Gamma(\mathcal{A}) \cong \mathcal{C}$, and $\Gamma(\mathcal{B})
\cong\mathcal{D}$. Notice also that here we abuse the notations: Formally speaking, we identify the two-element class $\{ \mathcal{A}, \mathcal{B}\}$ with the family containing all isomorphic copies of $\mathcal{A}$ and $\mathcal{B}$.

\subsection{Further background} \label{subsect:background}
	The Pullback Theorem and its consequences show that sometimes $tc$-em\-bed\-dings are too coarse: they \emph{cannot see}
finer structural distinctions between classes. One of the first examples of this phenomenon was provided by Chisholm, Knight, and
Miller~\cite{CKM07}: Let $VS$ be the class of infinite $\mathbb{Q}$-vector spaces, and let $ZS$ be the class of models of the
theory $\mathrm{Th}(\mathbb{Z},S)$, where $(\mathbb{Z},S)$ is the integers with successor. Then $VS$ and $ZS$ are equivalent with
respect to $tc$-em\-bed\-dings, but there is no computable embedding from $VS$ to $ZS$. 

Another example of this intriguing phenomenon was obtained by Ganchev, Ka\-li\-mul\-lin and Vatev~\cite{GKV}. For a structure $\A$, let $\tilde\A$ be the enrichment of $\A$ with a congruence relation $\sim$ such that every
congruence class in $\tilde\A$ is infinite and $\tilde\A /_\sim \cong \A$.
Then they showed that the class $\{\omega_S,\omega^\star_S\}$ is $tc$-equivalent to the class $\{\tilde\omega_S, \tilde\omega^\star_S\}$,
whereas $\{\tilde\omega_S, \tilde\omega^\star_S\}$ is not computably embeddable into $\{\omega_S,\omega^\star_S\}$.
Here $\omega_S$ and $\omega^\star_S$ are linear orderings of type $\omega$ and $\omega^\star$, respectively, together with the
successor relation. 

Our paper is focused on the degree $\deg_{tc}(\{ \omega,\omega^{\star}\})$. Historically speaking, the choice of this particular degree was motivated by the following open question:

\begin{problem*}[Kalimullin]
	It is easy to show that the pairs $\{ \omega, \omega^{\star}\}$ and $\{ \tilde\omega, \tilde\omega^{\star}\}$ are $tc$-equivalent. Moreover, $\{ \omega, \omega^{\star}\} \leq_c \{ \tilde\omega, \tilde\omega^{\star}\}$. Is there a computable embedding from $\{ \tilde\omega, \tilde\omega^{\star}\}$ to $\{ \omega, \omega^{\star}\}$?
\end{problem*}

One can attack the problem via employing model-theoretic properties of the structures (in a way similar to~\cite{CKM07}). In particular, a naive way to distinguish these pairs would be the following. Each of the orders $\omega$ and $\omega^{\star}$ is rigid, while both $\tilde\omega$ and $\tilde\omega^{\star}$ have continuum many automorphisms. Maybe, this fact can help us to prove that $\{ \tilde\omega, \tilde\omega^{\star}\} \nleq_c \{ \omega, \omega^{\star}\}$? Nevertheless, this is \emph{not} the case~--- one can show that $\{ \tilde\omega, \tilde\omega^{\star}\} \equiv_c \{ (\omega^2, B), (\omega\cdot\omega^{\star},B) \}$, where $B$ is the standard block relation on a linear order. Since the structures $(\omega^2, B)$ and $(\omega\cdot\omega^{\star},B)$ are both rigid, it seems that studying automorphism groups does not help in this setting. 

We note that recently, the theory of Turing computable embeddings found applications in algorithmic learning theory. Section~3.2 of~\cite{BFS-arxiv} establishes connections between $tc$-embeddings and a particular paradigm of learnability for classes of countable structures. Informally speaking, this paradigm employs a learner whose goal is, given the atomic diagram of a structure $\mathcal{A}$, to learn the isomorphism type of $\mathcal{A}$. The learner is allowed to use both positive and negative data provided by the atomic diagram. Remarkably, the family $\{\omega, \omega^{\star}\}$ is learnable by a computable learner. These facts lead to the following problem: it would be interesting to study connections between computable embeddings and algorithmic learning (specifically, topological aspects of learnability~--- see, e.g., \cite{dBY}).


\section{The $tc$-degree of $\{ \omega, \omega^{\star}\}$}\label{sect:tc-deg}

In this section, we give a characterization of the $tc$-degree for the class $\{\omega,\omega^{\star}\}$:

\begin{theorem}\label{theo:description}
	Let $\mathcal{A}$ and $\mathcal{B}$ be infinite $L$-structures such that $\mathcal{A}\not\cong\mathcal{B}$. Then the following conditions are equivalent:
	\begin{itemize}
		\item[(i)] $\{ \mathcal{A}, \mathcal{B}\} \equiv_{tc} \{ \omega, \omega^{\star}\}$.
		
		\item[(ii)] Both $\mathcal{A}$ and $\mathcal{B}$ are computably presentable, $\mathcal{A}\equiv_1 \mathcal{B}$, 
		\begin{equation}\label{equ:001-theories}
			\Sigma^c_2\text{-}Th(\mathcal{A}) \smallsetminus \Sigma^c_2\text{-}Th(\mathcal{B}) \neq \emptyset, \text{~and~} \Sigma^c_2\text{-}Th(\mathcal{B}) \smallsetminus \Sigma^c_2\text{-}Th(\mathcal{A}) \neq \emptyset.
		\end{equation}
	\end{itemize}
\end{theorem}

In order to prove the theorem, first, we establish the following useful fact:

\begin{proposition}\label{lem:turing_1+eta}
  Let $\mathcal{A}$ and $\mathcal{B}$ be infinite computable $L$-structures such that $\mathcal{A}\ncong\mathcal{B}$. If $\mathcal{A} \equiv_1 \mathcal{B}$, then $\{ \omega, \omega^{\star}\} \leq_{tc} \{ \mathcal{A},\mathcal{B}\}$.
\end{proposition}
\begin{proof} 
	We build a Turing operator $\Phi\colon \{ \omega, \omega^{\star}\} \leq_{tc} \{ \mathcal{A}, \mathcal{B}\}$.

	We may assume that $dom(\mathcal{A})=dom(\mathcal{B})=\mathbb{N}$. Choose computable sequences of finite structures $(\mathcal{A}^s)_{s\in\omega}$ and $(\mathcal{B}^s)_{s\in\omega}$ such that
	$\mathcal{A}^s\subset \mathcal{A}^{s+1}$, $\mathcal{B}^s\subset \mathcal{B}^{s+1}$, $dom(\mathcal{A}^s) = dom(\mathcal{B}^s) = \{0,1,\dots,s\}$,  $\bigcup_s \mathcal{A}^s = \mathcal{A}$, and $\bigcup_s \mathcal{B}^s = \mathcal{B}$.	 Note that here we use the following convention: If the language $L$ is infinite, then we fix a strongly computable sequence of finite languages $(L^s)_{s\in\omega}$ such that $\bigcup_{s\in\omega} L^s = L$ and $L^s \subset L^{s+1}$ for all $s$. The structures $\mathcal{A}^s$ and $\mathcal{B}^s$ are treated as $L^s$-structures.
        Moreover, since $\mathcal{A} \equiv_1 \mathcal{B}$, we can choose total computable functions $u(s)$ and $v(s)$ such that $\mathcal{A}^s$ can be embedded into $\mathcal{B}^{u(s)}$, and $\mathcal{B}^s$ can be embedded into $\mathcal{A}^{v(s)}$.
	
	Using the objects described above, one can construct total computable functions $\xi$ and $\psi$ acting from $2^{<\omega}$ to $2^{<\omega}$ such that:
	\begin{itemize}
		\item If a string $\sigma\in 2^{<\omega}$ encodes (the atomic diagram of) a finite structure $\mathcal{C}$ such that $\mathcal{C}\cong \mathcal{A}^s$ and $dom(\mathcal{C}) = \{0,1,\dots,s\}$, then $\xi(\sigma)$ encodes a finite structure $\mathcal{D}$ such that $\mathcal{C}\subseteq \mathcal{D}$, $\mathcal{D}\cong \mathcal{B}^{u(s)}$, and $dom(\mathcal{D}) = \{ 0, 1, \dots,u(s)\}$.
		
		\item The function $\psi$ satisfies a similar condition, where $\mathcal{A}$ and $\mathcal{B}$ need to switch places, and $u(s)$ is replaced by $v(s)$.
	\end{itemize}
	
	Consider a structure $\mathcal{S}$, with domain a subset of $\mathbb{N}$, which is isomorphic either to $\omega$ or to $\omega^{\star}$, but we do not know to which one.
        Assume that $dom(\mathcal{S}) = \{ d_0 <_{\mathbb{N}} d_1 <_{\mathbb{N}} d_2 <_{\mathbb{N}} \dots\}$, where $\leq_{\mathbb{N}}$ is the standard order of natural numbers.
        Our strategy for building the structure $\Phi(\mathcal{S})=\mathcal{C}$ is as follows:
        
        At a stage $s$, we define a finite structure $\mathcal{C}^s$ and a \emph{guess} $g[s] = (\mathcal{G}[s], t[s],$ $l[s], r[s])$, where $\mathcal{G}[s] \in \{ \mathcal{A}, \mathcal{B}\}$, $t[s]\in\omega$, and $l[s],r[s]\in dom(\mathcal{S})$. The meaning behind $g[s]$ is as follows. Assume that $g[s]$ is equal to, say, $(\mathcal{A}, 3, d_2,d_4)$. Then at the stage $s$ we have $\mathcal{C}^s\cong \mathcal{A}^3$. Since we are currently building $\mathcal{A}$, we think that $\mathcal{S}$ is isomorphic to $\omega$, and $d_2$ is the least element of $\mathcal{S}$. If our guess is equal to $(\mathcal{B}, 20, d_2,d_4)$, then our current beliefs are the following: We are building $\mathcal{B}$, $\mathcal{C}^s$ is a copy of $\mathcal{B}^{20}$, $\mathcal{S}\cong \omega^{\star}$, and $d_4$ is the greatest element in $\mathcal{S}$.
	
        At stage $0$, let $\mathcal{C}^0 = \mathcal{A}^0$, and our current guess $g[0]$ is equal to $(\mathcal{A},0,d_0,d_0)$.
	Consider a stage $s+1$ of the construction. Define $l[s+1]$ and $r[s+1]$ as the $\leq_{\mathcal{S}}$-least and the $\leq_{\mathcal{S}}$-greatest elements from the set $\{ d_0,d_1,\dots,d_{s+1}\}$, respectively. Consider two cases.
	
	\emph{Case 1.} Suppose that $\mathcal{G}[s] = \mathcal{A}$. If $l[s+1] <_{\mathcal{S}} l[s]$, then we switch to building $\mathcal{B}$: Using the function $\xi$, we find a structure $\mathcal{C}^{s+1}$ such that $\mathcal{C}^{s}\subseteq \mathcal{C}^{s+1}$ and $\mathcal{C}^{s+1} \cong \mathcal{B}^{u(t[s])}$. We define $\mathcal{G}[s+1] = \mathcal{B}$ and $t[s+1] =  u(t[s])$.
        
	If $l[s+1] \geq_{\mathcal{S}} l[s]$, then extend $\mathcal{C}^s$ to $\mathcal{C}^{s+1}$ such that $\mathcal{C}^{s+1}\cong \mathcal{A}^{t[s]+1}$ and $dom(\mathcal{C}^{s+1}) = \{0,1,\dots,t[s]+1\}$. Set $\mathcal{G}[s+1] = \mathcal{A}$ and $t[s+1] =  t[s] + 1$.
	
	\emph{Case 2.} Suppose that $\mathcal{G}[s] = \mathcal{B}$. If $r[s+1] >_{\mathcal{S}} r[s]$, then switch to building $\mathcal{A}$. Otherwise, continue building $\mathcal{B}$, mutatis mutandis.
	
	It is not hard to show that the described construction gives a Turing operator. If $\mathcal{S}$ is a copy of $\omega$, then consider the number $s_0$ such that $d_{s_0}$ is the $\leq_{\mathcal{S}}$-least element. If $\mathcal{G}[s_0]=\mathcal{A}$, then it is easy to see that $\mathcal{G}[t]=\mathcal{A}$ for all $t\geq s_0$, and $\Phi(\mathcal{S})$ is a copy of $\mathcal{A}$. If $\mathcal{G}[s_0]=\mathcal{B}$, then find the least $s_1>s_0$ such that $d_{s_1} >_{\mathcal{S}} r[s_0]$. Clearly, we have $\mathcal{G}[t]=\mathcal{A}$ for all $t\geq s_1$, and $\Phi(\mathcal{S})\cong \mathcal{A}$. If $\mathcal{S}$ is a copy of $\omega^{\star}$, then, mutatis mutandis, $\Phi(\mathcal{S})$ is isomorphic to $\mathcal{B}$. Proposition~\ref{lem:turing_1+eta} is proved.
\end{proof}

\begin{proof}[Proof of Theorem~\ref{theo:description}] \emph{(i)$\Rightarrow$(ii).} Since $\{ \omega, \omega^{\star}\} \leq_{tc}\{ \mathcal{A}, \mathcal{B}\}$, both $\mathcal{A}$ and $\mathcal{B}$ have computable copies. 
	
	Suppose that $\mathcal{A} \not\equiv_1 \mathcal{B}$. W.l.o.g., one may assume that there is an $\exists$-sentence $\psi$ which is true in $\mathcal{A}$, but not true in $\mathcal{B}$. By applying Pullback Theorem to the reduction $\Phi\colon \{ \omega, \omega^{\star}\} \leq_{tc}\{ \mathcal{A}, \mathcal{B}\}$, we obtain a $\Sigma^c_1$ sentence $\psi^{\star}$ such that $\psi^{\star}$ is true in $\omega$ and false in $\omega^{\star}$. This gives a contradiction, since $\omega \equiv_1 \omega^{\star}$.
	Hence, $\mathcal{A} \equiv_1 \mathcal{B}$.
	
	Consider $\exists\forall$-sentences $\xi_0$ and $\xi_1$ in the language of linear orders which say the following: ``there is a least element'' and ``there is a greatest element,'' respectively. Since $\{\mathcal{A},\mathcal{B}\} \leq_{tc} \{ \omega, \omega^{\star}\}$, one can apply Pullback Theorem to $\xi_0$ and $\xi_1$, and obtain condition~(\ref{equ:001-theories}).

	\emph{(ii)$\Rightarrow$(i).} Proposition~\ref{lem:turing_1+eta} implies that $\{ \omega, \omega^{\star}\} \leq_{tc} \{\mathcal{A},\mathcal{B}\}$. Now we need to build a Turing operator $\Phi\colon \{ \mathcal{A},\mathcal{B}\}\leq_{tc} \{\omega,\omega^{\star}\}$. 
	
	Fix $\Sigma^c_2$ sentences $\varphi$ and $\psi$ such that $\mathcal{A}\models \varphi\ \&\ \neg\psi$ and $\mathcal{B}\models \neg\varphi\ \&\ \psi$. W.l.o.g., one may assume that
	\[
		\varphi = \exists \bar x \bigwedge_{i\in \omega} \forall \bar y_i \varphi_i(\bar x,\bar y_i), \ \ \ \psi = \exists \bar u \bigwedge_{j\in \omega} \forall \bar v_j \psi_j(\bar u,\bar v_j).
	\]
	
	Let $\mathcal{S}$ be a copy of one of the structures $\mathcal{A}$ or $\mathcal{B}$. We give an informal description of how to build the structure $\Phi(\mathcal{S})$. Formal details can be recovered from the proof of Proposition~\ref{lem:turing_1+eta}, mutatis mutandis.
	
	Suppose that $dom(\mathcal{S}) = \{ d_0 <_{\mathbb{N}} d_1 <_{\mathbb{N}} d_2 <_{\mathbb{N}} \dots\}$, where $\leq_{\mathbb{N}}$ is the standard order of natural numbers. By $dom(\mathcal{S})[s]$ we denote the set $\{d_0,d_1,\dots, d_s\}$. We say that a tuple $\bar d$ from $dom(\mathcal{S})[s]$ is a \emph{$\varphi[s]$-witness} if for any $i\leq s$ and any tuple $\bar y_i$ from $dom(\mathcal{S})[s]$, we have $\mathcal{S}\models \varphi_i(\bar d,\bar y_i)$. The notion of a \emph{$\psi[s]$-witness} is defined in a similar way. The order of witnesses is induced by their G{\"o}del numbers.
	
	At a stage $s+1$, we consider the following four cases.
	
	\emph{Case 1.} There are no $\varphi[s+1]$-witnesses and no $\psi[s+1]$-witnesses. Then extend $\Phi(\mathcal{S})[s]$ to $\Phi(\mathcal{S})[s+1]$ by copying (a finite part of) $\omega$. In particular, put into $\Phi(\mathcal{S})[s+1]$ an element which is $\leq_{\Phi(\mathcal{S})}$-greater than every element from $\Phi(\mathcal{S})[s]$.
	
	\emph{Case 2.} There is a $\varphi[s+1]$-witness, and there are no $\psi[s+1]$-witnesses. Proceed as in the previous case.
	
	\emph{Case 3.} There is a $\psi[s+1]$-witness, and there are no $\varphi[s+1]$-witnesses. Extend $\Phi(\mathcal{S})[s]$ to $\Phi(\mathcal{S})[s+1]$ by copying $\omega^{\star}$.
	
	\emph{Case 4.} There are both $\varphi[s+1]$-witnesses and $\psi[s+1]$-witnesses. If the least $\varphi[s+1]$-witness is $\leq_{\mathbb{N}}$-less than the least $\psi[s+1]$-witness, then copy $\omega$. Otherwise, copy $\omega^{\star}$.
	
	It is not difficult to show that the construction gives a Turing operator $\Phi$ with the following properties. If $\mathcal{S}$ is a copy of $\mathcal{A}$, then $\Phi(\mathcal{S})\cong \omega$. If $\mathcal{S}\cong \mathcal{B}$, then $\Phi(\mathcal{S})\cong \omega^{\star}$. Theorem~\ref{theo:description} is proved.
\end{proof}


\section{The top pair}
\label{sect:top:pair}

Our goal in this section is to prove that among all pairs of linear orders, which are $tc$-equivalent to $\{\omega,\omega^\star\}$, there is a greatest pair
under computable embeddings, namely $\{1+\eta, \eta + 1\}$.

Let us denote by $\mathcal{E}$ the equivalence structure with infinitely many equivalence classes of infinite size and no classes of other size.
By $\mathcal{E}_k$ we will denote the equivalence structure with infinitely many equivalence classes of infinite size and {\em exactly one} equivalence class of size $k$,
and $\hat{\mathcal{E}}_k$ denotes the equivalence structure with infinitely many equivalence classes of infinite size and {\em infinitely many} equivalence classes of size $k$.
It is straightforward to see that $\{\mathcal{E}_1,\mathcal{E}_2\} \leq_c \{\hat{\mathcal{E}}_1, \hat{\mathcal{E}}_2\}$.

\begin{proposition}
  \label{prop:order-to-equivalence}
  Let $\mathcal{L}_1$ and $\mathcal{L}_2$ be two linear orders such that 
  $\mathcal{L}_1$ has a least element and no greatest element, and $\mathcal{L}_2$ has a greatest element and no least element.
  Then $\{\mathcal{L}_1, \mathcal{L}_2\} \leq_c \{\mathcal{E}_1,\mathcal{E}_2\}$.
\end{proposition}
\begin{proof}
  Suppose that the input structure $\mathcal{S}$ have domain $\{x_i \mid i < \omega\}$.
  The output structure will have domain a subset of $\{y_{i,j} \mid i,j < \omega \}$.
  The enumeration operator $\Gamma$ works as follows.
  On input the finite atomic diagram
  \[x_{k_0} <_{\mathcal{S}} x_{k_1} <_{\mathcal{S}}\cdots <_{\mathcal{S}} x_{k_n},\] it outputs an \emph{infinite} part of the basic diagram of the  output structure saying 
  \begin{itemize}
  \item 
    The class ${[y_{k_0,0}]}_{\sim}$ has size at least one witnessed by the element $y_{k_0,0}$;
  \item
    The class ${[y_{k_n,0}]}_\sim$ has size at least two witnessed by $y_{k_n,0}$ and $y_{k_n,1}$;
  \item
    For $i=1,\dots,n-1$, the classes ${[y_{k_i,0}]}_{\sim} = \{y_{k_i,j} \mid j < \omega \}$.
  \end{itemize}
  If the input structure has a $<_{\mathcal{S}}$-least element, say $x_{k_0}$, then the equivalence class of $y_{k_0,0}$ will stay with exactly one element,
  and $\Gamma$ will enumerate infinitely many elements in the equivalence class of $y_{k_n,0}$ when it sees as input
  a finite atomic diagram of the form
  $x_{k_0} <_{\mathcal{S}} x_{k_1} <_{\mathcal{S}} \cdots <_{\mathcal{S}} x_{k_n} <_{\mathcal{S}} x_{k_{n+1}}$.
  Similarly, if the input structure has a $<_{\mathcal{S}}$-greatest element, say $x_{k_n}$, then the equivalence class of
  $y_{k_n,0}$ will stay with exactly two elements,
  and $\Gamma$ will enumerate infinitely many elements in the equivalence class of $y_{k_0,0}$ when it sees as input
  a finite atomic diagram of the form
  $x_{k_{n+1}} <_{\mathcal{S}} x_{k_0} <_{\mathcal{S}} x_{k_1} <_{\mathcal{S}} \cdots <_{\mathcal{S}} x_{k_n}$.
\end{proof}
By transitivity of $\leq_c$, we immediately get the following corollary.
\begin{corollary}
  Let $\mathcal{L}_1$ and $\mathcal{L}_2$ be two linear orders such that 
  $\mathcal{L}_1$ has a least element and no greatest element, and $\mathcal{L}_2$ has a greatest element and no least element.
  Then $\{\mathcal{L}_1,\mathcal{L}_2\} \leq_c \{\hat{\mathcal{E}}_1, \hat{\mathcal{E}}_2\}$.
\end{corollary}

In the end of this section, we will need the following special cases.
\begin{corollary}
  $ \{1 + \eta,\eta + 1\} \leq_c \{\mathcal{E}_1, \mathcal{E}_2\}$ and
  $\{1 + \eta,\eta + 1\} \leq_c \{\hat{\mathcal{E}}_1, \hat{\mathcal{E}}_2\}$.
\end{corollary}

The next result is not so useful in itself, because isomorphic copies of the input structure produce non-isomorphic
copies of the output structure, but when we replace every element of the output structure by a copy of $\eta$,
we will get isomorphic structures.

\begin{proposition}
  \label{prop:equivalence:reduction:1}
  Let $\mathcal{K}_1$ be the class of linear orders, which have a least element and no greatest element,
  and let $\mathcal{K}_2$ be the class of linear orders, which have no least element and no greatest element.
  Then $\{\hat{\mathcal{E}}_1, \hat{\mathcal{E}}_2\} \leq_c \{\mathcal{K}_1,\mathcal{K}_2\}$,
  which means that there is an enumeration operator $\Gamma$ such that
  for any copy $\hat{\mathcal{S}}_i$ of $\hat{\mathcal{E}}_i$, $\Gamma(\hat{\mathcal{S}}_i) \in \mathcal{K}_i$, for $i = 0,1$.
\end{proposition}
\begin{proof}
  Given as input a structure $\mathcal{S}$ in the language of equivalence structures, the enumeration operator $\Gamma$ will output a linear order with domain $D$
  consisting of tuples of elements from $\mathcal{S}$, where
  \[D = \{(x_0,\dots,x_n) \mid \bigwedge_{i<n}(x_i <_{\mathbb{N}} x_{i+1}\ \&\ |[x_i]_\sim| \geq 2)\},\]
  and for two such tuples $\bar{x} = (x_0,\dots,x_n)$ and $\bar{y} = (y_0,\dots,y_k)$, we will say that $\bar{x} \prec \bar{y}$ if
  \begin{itemize}
  \item
    $\bar{x}$ is a proper extension of $\bar{y}$;
  \item
    otherwise, if the first index where $\bar{x}$ and $\bar{y}$ differ is $i$, then $x_i < y_i$.
  \end{itemize}
  
  Now we consider the linear orders $\mathcal{L}_1 = \Gamma(\hat{\mathcal{E}}_1)$ and $\mathcal{L}_2 = \Gamma(\hat{\mathcal{E}}_2)$.
  First we will show that $\mathcal{L}_1 \in \mathcal{K}_1$.
  Given the structure $\hat{\mathcal{E}}_1$, let $\hat{x}$ be the least element in the domain of $\hat{\mathcal{E}}_1$, in the order of natural numbers, describing an equivalence class with exactly one element in $\hat{\mathcal{E}}_1$ and let $\bar{x} = (x_0,\dots,x_n)$ be all elements of $\hat{\mathcal{E}}_1$, ordered as natural numbers, less that $\hat{x}$ such that
  \[x_0 <_{\mathbb{N}} x_1 <_{\mathbb{N}} \cdots <_{\mathbb{N}} x_{n-1} <_{\mathbb{N}} x_n = \hat{x}.\]
  It is easy to see that $\bar{x}$ cannot have proper extensions in the domain of $\mathcal{L}_1$ because the equivalence class of $x_n$ has size one.
  Then if $(y_0,\dots,y_k) \prec (x_0,\dots,x_n)$, we have $y_i < x_i$, for some $i$, but this is impossible since $\bar{x}$ contains all elements less that $\hat{x}$ ordered in a strictly increasing order.
  We conclude that $\bar{x}$ is the least element of $\mathcal{L}_1$.
  It is easy to see that $\mathcal{L}_1$ does not have a greatest element. It follows that $\mathcal{L}_1 \in \mathcal{K}_1$.

  Similarly, it is clear that $\mathcal{L}_2 = \Gamma(\hat{\mathcal{E}}_2)$ does not have a greatest element.
  We will show that $\mathcal{L}_2$ does not have a least element.
  Consider an arbitrary $\bar{x} = (x_0,\dots,x_n)$ in the domain of $\mathcal{L}_2$.
  Define $\bar{x}' = (x_0,\dots,x_n,x')$, for some $x'$ such that $x_n <_{\mathbb{N}} x'$. Then $\bar{x}' \prec \bar{x}$, because $\bar{x}'$ is a proper extension of $\bar{x}$. 
\end{proof}

\begin{corollary}
  \label{cor:equivalence:eta:1}
  $\{\hat{\mathcal{E}}_1,\hat{\mathcal{E}}_2\} \leq_c \{1+\eta, \eta\}$.
\end{corollary}
\begin{proof}
  Given the enumeration operator $\Gamma$ from the proof of Proposition \ref{prop:equivalence:reduction:1},
  we produce a new enumeration operator, where every element of $\Gamma(\hat{\mathcal{E}}_i)$ is replaced by an interval of rational numbers of the form $[p,q)$.
\end{proof}

\begin{proposition}
  \label{prop:equivalence:reduction:2}
  Let $\mathcal{K}_1$ be the class of infinite linear orders,  which have a least element and no greatest element,
  and let $\mathcal{K}_2$ be the class of infinite linear orders, which have no least element and no greatest element.
  Then $\{\hat{\mathcal{E}}_1, \hat{\mathcal{E}}_2\} \leq_c \{\mathcal{K}_2,\mathcal{K}_1\}$.
\end{proposition}
\begin{proof}
  Given as input a structure $\mathcal{S}$ in the language of equivalence structures, the enumeration operator $\Gamma$ will output a linear order with domain $D$
  consisting of tuples of elements from $\mathcal{S}$, where
  \[D = \{(x_0,\dots,x_n) \mid \bigwedge_{i<n} (x_i <_{\mathbb{N}} x_{i+1}\ \&\ |[x_i]_\sim| \geq 3)\ \&\ |[x_n]_\sim| \geq 2 \}.\]
  Now we essentially repeat the proof of Proposition \ref{prop:equivalence:reduction:1}.
\end{proof}

\begin{corollary}
  \label{cor:equivalence:eta:2}
  $\{\hat{\mathcal{E}}_1,\hat{\mathcal{E}}_2\} \leq_c \{\eta, \eta+1\}$.
\end{corollary}
\begin{proof}
  First we reverse the relation in the construction of $\Gamma$ from Proposition \ref{prop:equivalence:reduction:2} to produce a linear order with a greatest element.
  Then we produce a new enumeration operator, where every element of the linear order is replaced by an interval of rational numbers of the form $(p,q]$.
\end{proof}

\begin{corollary}
  \label{cor:equivalence-to-eta}
  $\{\hat{\mathcal{E}}_1,\hat{\mathcal{E}}_2\} \leq_c \{1+\eta, \eta+1\}$.
\end{corollary}
\begin{proof}
  We concatenate the results of the enumeration operators $\Gamma_0$ from Corollary \ref{cor:equivalence:eta:1} and $\Gamma_1$ from Corollary \ref{cor:equivalence:eta:2}. More formally, given an input structure $\mathcal{A}$, our concatenation operator $\Delta$ copies $\Gamma_0(\mathcal{A})$ on the set of even numbers, and $\Delta$ copies $\Gamma_1(\mathcal{A})$ on the odd numbers. In addition, $\Delta$ declares that in the output structure, every even number is strictly less than every odd number. We observe that $1 + \eta + \eta = 1 + \eta$ and $\eta + \eta + 1 = \eta + 1$; hence, $\Delta$ has the desired properties.
\end{proof}

\begin{theorem}
  \label{th:eta:equivalence}
  $\{1 + \eta, \eta + 1\} \equiv_c \{\hat{\mathcal{E}}_1, \hat{\mathcal{E}}_2\} \equiv_c \{\mathcal{E}_1, \mathcal{E}_2\}$.
\end{theorem}
\begin{proof}
  By Corollary \ref{cor:equivalence-to-eta}, we have that $\{\mathcal{E}_1, \mathcal{E}_2\} \leq_c \{\hat{\mathcal{E}}_1, \hat{\mathcal{E}}_2\} \leq_c \{1 + \eta, \eta + 1\}$
  and by Proposition \ref{prop:order-to-equivalence} we have that $\{1 + \eta, \eta + 1\} \leq_c \{\mathcal{E}_1, \mathcal{E}_2\}$.
\end{proof}

Recall that the structure $\mathcal{E}$ has infinitely many classes of infinite size, and no classes of other sizes.

\begin{proposition}
  \label{prop:formula-equivalence}
  Suppose that $\mathcal{A}$ and $\mathcal{B}$ are structures in the same language, for which there exists a $\Sigma^c_2$ sentence $\phi$ such that
  $\mathcal{A} \models \phi$ and $\mathcal{B} \models \neg \phi$.
  Then $\{\mathcal{A}, \mathcal{B}\} \leq_c \{\hat{\mathcal{E}}_1,\mathcal{E}\}$.
\end{proposition}
\begin{proof}
  Suppose that we have an effective listing
  $\{\alpha_{i,j}(x,y)\}_{i,j<\omega}$ of atomic formulas such that 
  \[\mathcal{A} \models \bigvee_i \exists x \bigwedge_j \forall y \alpha_{i,j}(x,y) \mbox{ and } \mathcal{B} \models \bigwedge_i \forall x \bigvee_j \exists y \neg \alpha_{i,j}(x,y).\]

  Suppose we are given an arbitrary structure $\mathcal{C}$ in the language of $\mathcal{A}$ and $\mathcal{B}$.
  We describe an enumeration operator $\Gamma$ which is provided with $\mathcal{C}$ as input.
  For every $c \in dom(\mathcal{C})$ the operator $\Gamma$ enumerates $c_{\langle{i,0}\rangle}$, for $i<\omega$, in the output structure $\Gamma(\mathcal{C})$, together
  with basic sentences saying that all of these $c_{\langle{i,0}\rangle}$ are in different equivalence classes.
  Moreover, if $\Gamma$ sees in the input structure a basic sentence of the form $\neg \alpha_{i,j}(c,d)$, it enumerates in the output structure
  the elements $c_{\langle{i,k}\rangle}$, for $k < \omega$, and basic sentences saying that they all belong to the equivalence class of $c_{\langle{i,0}\rangle}$.

  If $\mathcal{C} \cong \mathcal{A}$, then there is some $i < \omega$ and $c \in \mathcal{C}$ such that
  $\mathcal{C} \models \bigwedge_j \forall y \alpha_{i,j}(c,y)$.
  Thus, $c_{\langle{i,0}\rangle}$ forms an equivalence class with exactly one element.
  It follows that $\Gamma(\mathcal{A})$ produces an equivalence structure with \emph{at least} one equivalence class with exactly one element.

  If $\mathcal{C} \cong \mathcal{B}$, then for every natural number $i$ and element $c \in \mathcal{C}$, there is
  some $j$ and $d$ such that $\mathcal{C} \models \neg \alpha_{i,j}(c,d)$.
  By the construction of $\Gamma$, it follows that all $c_{\langle{i,k}\rangle}$, for $k < \omega$, form an equivalence class with infinitely many elements.
  Then $\Gamma(\mathcal{B})$ produces an equivalence structure in which every equivalence class contains infinitely many elements.

  Now it is easy to modify $\Gamma$ so that $\Gamma(\mathcal{A}) \cong \hat{\mathcal{E}}_1$ and $\Gamma(\mathcal{B}) \cong \mathcal{E}$.
  The new modified $\Gamma$ simply produces infinitely many copies of each equivalence class.
\end{proof}

\begin{corollary}
  \label{cor:formulas-equivalence}
  Suppose that $\mathcal{A}$ and $\mathcal{B}$ are structures for which there exist $\Sigma^c_2$ sentences $\phi$ and $\psi$ such that
  $\mathcal{A} \models \phi\ \&\ \neg\psi$ and $\mathcal{B} \models \neg \phi\ \&\ \psi$. Then
  \[\{\mathcal{A}, \mathcal{B}\} \leq_c \{\hat{\mathcal{E}}_1,\hat{\mathcal{E}}_{2}\}.\] 
\end{corollary}
\begin{proof}
  For the formula $\phi$, we apply Proposition \ref{prop:formula-equivalence} and produce an enumeration operator $\Gamma_1$ such that
  $\{\mathcal{A},\mathcal{B}\} \leq_c \{\hat{\mathcal{E}}_1,\mathcal{E}\}$.
  It is trivial to modify the proof of Proposition \ref{prop:formula-equivalence} and apply it for the formula $\psi$ to obtain
  an enumeration operator $\Gamma_2$ such that
  $\{\mathcal{A},\mathcal{B}\} \leq_c \{\mathcal{E},\hat{\mathcal{E}}_2\}$.
  Then we combine the two operators into one by simply taking a disjoint union of their outputs.
\end{proof}

\begin{theorem}
  The pair $\{1+\eta, \eta+1\}$ is the greatest one under computable embedding in the $tc$-equivalence class of the pair $\{\omega,\omega^\star\}$.
\end{theorem}
\begin{proof}
  Consider $\{\mathcal{A},\mathcal{B}\} \equiv_{tc} \{\omega,\omega^\star\}$.
  By Theorem~\ref{theo:description}, there exist $\Sigma^c_2$ sentences $\phi$ and $\psi$ such that $\mathcal{A} \models \phi\ \&\ \neg \psi$ and $\mathcal{B} \models \neg\phi\ \&\ \psi$.
  By combining Corollary \ref{cor:formulas-equivalence} with Theorem \ref{th:eta:equivalence}, we conclude that
  $\{\mathcal{A},\mathcal{B}\} \leq_{c} \{1 + \eta , \eta + 1 \}$.
\end{proof}

\section{An infinite chain of pairs}
\label{sect:infinite:chain}

In this section we work only with structures in the language of linear orders.
We denote by $\alpha$, $\beta$, $\gamma$ finite linear orders.
For an enumeration operator $\Gamma$, a finite linear order $\alpha$,
and an atomic formula $\varphi(\bar{x})$, and a tuple $\bar{a}$, we define
\[\alpha \Vdash_\Gamma \phi(\bar{a})\ \stackrel{\text{def}}{\iff}\ \bar{a} \in \Gamma(\alpha)\ \&\ \neg(\exists \beta \supseteq \alpha)[\  \Gamma(\beta) \models \neg\phi(\bar{a})\ ].\]
Here we write $\bar{a} \in \Gamma(\alpha)$ for $\Gamma(\alpha) \models \bigwedge^n_{i=1} a_i = a_i$, or in other
words, the sentences $a_i = a_i$ are enumerated by $\Gamma(\alpha)$.

Moreover, we will say that $\alpha$ {\em decides} $x$ and $y$ if $\Gamma(\alpha)\models x < y$ or $\Gamma(\alpha) \models y \leq x$.
For two finite linear orders $\alpha$ and $\beta$ with disjoint domains, we will write $\alpha + \beta$ for the finite linear order
obtained by merging $\alpha$ and $\beta$ so that the greatest element of $\alpha$ is less than the least element of $\beta$.
Following Rosenstein~\cite{Ros82}, we will use the notation $\sum_{i\in\omega} \alpha_i$ for the linear order $\alpha_0 + \alpha_1 + \cdots + \alpha_n + \cdots$,
and the notation $\sum_{i\in\omega^\star} \alpha_i$ for the linear order $\cdots + \alpha_n + \cdots + \alpha_1 + \alpha_0$.

\begin{proposition}
  \label{prop:force:either}
  Suppose $x,y \in \Gamma(\alpha)$ and $x \neq y$. Then
  \[\alpha \Vdash_\Gamma x < y \text{ or } \alpha \Vdash_\Gamma y < x.\]
\end{proposition}
\begin{proof}
  Towards a contradiction, assume that
  \[\alpha \not\Vdash_\Gamma x < y\text{ and }\alpha \not\Vdash_\Gamma y < x.\]
  This means that there is some $\alpha' \supset \alpha$ such that $\Gamma(\alpha') \models y < x$ and
  some $\alpha'' \supset \alpha$ such that $\Gamma(\alpha'') \models x < y$.
  Since $x,y \in \Gamma(\alpha)$, for any copy $\A$ of $\omega$ which extends $\alpha$ we have $\Gamma(\A) \models x < y$ or $\Gamma(\A) \models y < x$.
  Choose such a copy $\A$ of $\omega$ so that $\A \cap \alpha' = \A \cap \alpha'' = \alpha$.
  By compactness of enumeration operators, there is some finite part $\beta$ of $\A$, such that $\alpha'\cap \beta = \alpha'' \cap \beta = \alpha$, which decides $x$ and $y$.
  Without loss of generality, suppose that $\Gamma(\beta) \models x < y$.
  Then let $\gamma = \beta \cup \alpha'$. By monotonicity of $\Gamma$,
  since $\alpha' \subset \gamma$, we have $\Gamma(\gamma)\models y < x$ and since $\beta \subset \gamma$, we have $\Gamma(\gamma) \models x < y$.
  We reach a contradiction.
\end{proof}

\begin{corollary}
  \label{cor:force:either:general}
  Let $x_0,\dots,x_n$ be distinct elements and $x_0,\dots,x_n \in \Gamma(\alpha)$.
  Then there is some permutation $\pi$ of $(0,\dots,n)$ such that
  \[ \alpha \Vdash_\Gamma x_{\pi(0)} < x_{\pi(1)} < \cdots < x_{\pi(n)}.\]
\end{corollary}
\begin{proof}
  Assume that for all permutations $\pi$,
  \[ \alpha \not\Vdash_\Gamma x_{\pi(0)} < x_{\pi(1)} < \cdots < x_{\pi(n)}.\]
  Consider two distinct permutations which produce linear orders which differ at some two positions, say
  $i$ and $j$. Then we have
  \[\alpha \not\Vdash_\Gamma x_i < x_j\text{ and }\alpha \not\Vdash_\Gamma x_j < x_i.\]
  Now we apply Proposition \ref{prop:force:either} and we reach a contradiction.
\end{proof}

\begin{proposition}
  \label{prop:force:excluded-middle}
  Suppose $x, y \in \Gamma(\alpha)$ and $x \neq y$. Then
  \[\alpha \Vdash_\Gamma x < y \iff \alpha \not\Vdash_\Gamma y < x.\]
\end{proposition}
\begin{proof}
  $(1) \to (2)$. Since $x,y \in \Gamma(\alpha)$, there is some $\beta \supseteq \alpha$ which decides $x$ and $y$.
  Since $\alpha \Vdash_\Gamma x < y$, it follows that we must have $\Gamma(\beta)\models x < y$ and hence $\alpha \not \Vdash_\Gamma y < x$.

  $(2) \to (1)$. Suppose that $\alpha \not\Vdash_\Gamma y < x$. By Proposition \ref{prop:force:either}, we have $\alpha \Vdash_\Gamma x < y$.
\end{proof}

\begin{corollary}\label{cor:force:excluded-middle:general}
  Let $x_0,\dots,x_n$ be distinct elements and $x_0,\dots,x_n \in \Gamma(\alpha)$.
  Then there exists \emph{exactly one} permutation $\pi$ of $(0,\dots,n)$ for which
  \[ \alpha \Vdash_\Gamma x_{\pi(0)} < x_{\pi(1)} < \cdots < x_{\pi(n)}.\]
\end{corollary}
\begin{proof}
  By Corollary \ref{cor:force:either:general}, there is some permutation $\pi$ such that
  \[ \alpha \Vdash_\Gamma x_{\pi(0)} < x_{\pi(1)} < \cdots < x_{\pi(n)}.\]
  Assume that there is another permutation $\rho$ for which
  \[ \alpha \Vdash_\Gamma x_{\rho(0)} < x_{\rho(1)} < \cdots < x_{\rho(n)}.\]
  Let us say that these two permutation produce linear orders which differ at positions $i$ and $j$.
  It follows that
  \[ \alpha \Vdash_\Gamma x_i < x_j\text{ and } \alpha \Vdash_\Gamma x_j < x_i.\]
  Now we reach a contradiction with Proposition \ref{prop:force:excluded-middle}.
\end{proof}

\begin{proposition}
  \label{prop:alpha:finite}
  Suppose that $\Gamma : \{\mathcal{A},\mathcal{B}\} \leq_c \{\mathcal{C},\mathcal{D}\}$, where
  $\mathcal{C}$ has no infinite descending chains, and $\mathcal{D}$ has no infinite ascending chains.
  Then $\Gamma(\alpha)$ is finite for any finite linear order $\alpha$.
\end{proposition}
\begin{proof}
  Towards a contradiction, assume the opposite and choose distinct elements $x_i \in \Gamma(\alpha)$, for $i < \omega$.
  By Corollary~\ref{cor:force:excluded-middle:general}, we have either 
  \[ (\forall i < \omega)[\ \alpha \Vdash_\Gamma x_{i} < x_{i+1}\ ] \text{ or } (\forall i < \omega)[\ \alpha \Vdash_\Gamma x_{i+1} < x_i\ ].\]
  In the first case, we extend $\alpha$ to a copy $\hat{\mathcal{B}}$ of $\mathcal{B}$. Then
  $\Gamma(\hat{\mathcal{B}}) \models \bigwedge_{i<\omega} x_i < x_{i+1}$,
  which is a contradiction. In the second case, we extend $\alpha$ to a copy $\hat{\mathcal{A}}$ of $\mathcal{A}$ and again reach a contradiction.
\end{proof}

For any $\Gamma$ satisfying the conditions of Proposition~\ref{prop:alpha:finite}, there are infinitely many finite linear orders $\alpha$ such that $\Gamma(\alpha)\neq \emptyset$.
In what follows, we will always suppose that we consider only such finite linear orders $\alpha$.

\begin{proposition}
  \label{prop:alpha:beta}
  Suppose that $\alpha$ and $\beta$ are finite linear orders with disjoint domains,
  and $x,y \in \Gamma(\alpha) \cap \Gamma(\beta)$, $x \neq y$. Then
  \[\alpha \Vdash_\Gamma x < y\ \text{iff}\ \beta \Vdash_\Gamma x < y.\]
\end{proposition}
\begin{proof}
  Let $\alpha \Vdash_\Gamma x < y$ and, by Proposition \ref{prop:force:either}, assume that $\beta \Vdash_\Gamma y < x$.
  Let $\gamma$ be a finite linear order such that $\alpha,\beta \subseteq \gamma$.
  By monotonicity, since $\alpha \Vdash_\Gamma x < y$, then $\gamma \Vdash_\Gamma x < y$, and
  since $\beta \Vdash_\Gamma y < x$, then $\gamma \Vdash_\Gamma y < x$.
  We reach a contradiction by Proposition \ref{prop:force:excluded-middle}.
\end{proof}

\begin{proposition}\label{prop:finite:intersection}
  Suppose that $\Gamma : \{\mathcal{A},\mathcal{B}\} \leq_c \{\mathcal{C},\mathcal{D}\}$, where
  $\mathcal{A}$, $\mathcal{C}$ have no infinite descending chains, and $\mathcal{B}$, $\mathcal{D}$
  have no infinite ascending chains.
  There are at most finitely many elements $x$ with the property that there exist $\alpha$ and $\beta$ with disjoint domains and
  $x \in \Gamma(\alpha) \cap \Gamma(\beta)$.
\end{proposition}
\begin{proof}
  Towards a contradiction, assume the opposite.
  By Proposition~\ref{prop:alpha:finite}, there is an infinite sequence of mutually disjoint finite linear orders $\alpha_i$ and $\beta_i$, $i < \omega$,
  and distinct elements $x_i \in \Gamma(\alpha_i) \cap \Gamma(\beta_i)$.
  Consider a copy $\mathcal{\hat{A}}$ of $\mathcal{A}$ extending $\sum_{i\in\omega} \alpha_i$.
  Clearly every $x_i \in \Gamma(\mathcal{\hat{A}})$ and hence
  \[\Gamma(\mathcal{\hat{A}}) \models \bigwedge_{i<\omega} x_{\pi(i)} < x_{\pi(i+1)},\]
  for some permutation $\pi$ of $\omega$.
  For simplicity, suppose that $\pi$ is the identity function.
  Then we have that for any index $i$, $\alpha_i + \alpha_{i+1} \Vdash_\Gamma x_i < x_{i+1}$.
  Now, since $x_i \in \Gamma(\beta_i)$ and $x_{i+1} \in \Gamma(\beta_{i+1})$, by Proposition~\ref{prop:alpha:beta}, we have
  $\beta_{i+1} + \beta_i \Vdash_\Gamma x_i < x_{i+1}$.
  In this way we can build a copy $\mathcal{\hat{B}}$ of $\mathcal{B}$ extending $\sum_{i\in\omega^\star} \beta_i$,
  and obtain
  \[\Gamma(\mathcal{\hat{B}}) \models \bigwedge_{i<\omega} x_i < x_{i+1},\]
  which is a contradiction, because $\Gamma(\hat{\mathcal{B}})$ is a copy of $\mathcal{D}$, which has no infinite ascending chains.
\end{proof}

Let us call $(x,\alpha)$ a $\Gamma$-pair if $x \in \Gamma(\alpha)$.
In view of Proposition~\ref{prop:finite:intersection}, for any sequence of $(\alpha_i)_{i<\omega}$ such that $\Gamma(\alpha_i) \neq \emptyset$,
we may assume that there is an infinite subsequence of $\Gamma$-pairs $(x_i,\alpha_{k_i})_{i<\omega}$, where all $x_i$ are distinct elements.

\begin{proposition}
  \label{prop:ascending:chain}
  For any two sequences of $\Gamma$-pairs $(x_i,\alpha_i)_{i\in\omega}$ and $(y_i,\beta_i)_{i\in\omega}$, the following are equivalent:
  \begin{itemize}
  \item[(i)]
    $\Gamma(\sum_{i\in\omega} \alpha_i + \sum_{i\in\omega}\beta_i) \models \bigwedge_{i<\omega} x_i < y_i < x_{i+1}$;
  \item[(ii)]
    $\Gamma(\sum_{i\in\omega^\star}\alpha_i + \sum_{i\in\omega^\star}\beta_i) \models \bigwedge_{i<\omega} x_i < y_i < x_{i+1}$.
  \end{itemize}
\end{proposition}
\begin{proof}
  The two directions are symmetrical. Without loss of generality, suppose that
  \begin{equation}
    \label{eq:1}
    \Gamma(\sum_{i\in\omega} \alpha_i + \sum_{i\in\omega}\beta_i) \models \bigwedge_{i<\omega} x_i < y_i < x_{i+1}.
  \end{equation}
  It is enough to show that for an arbitrary $i$,
  \[\alpha_{i+1} + \alpha_{i} + \beta_{i+1} + \beta_{i} \Vdash_\Gamma x_i < y_i < x_{i+1} < y_{i+1}.\]
  Since $x_i \in \Gamma(\alpha_{i})$ and $y_i \in \Gamma(\beta_{i})$, by the monotonicity of $\Gamma$ and (\ref{eq:1}), we have
  $\alpha_{i} + \beta_{i} \Vdash_\Gamma x_i < y_i$.
  Similarly, we have
  \[\alpha_{i+1} + \beta_{i} \Vdash_\Gamma y_i < x_{i+1}\mbox{ and } \alpha_{i+1} + \beta_{i+1} \Vdash_\Gamma x_{i+1} < y_{i+1}.\]

  Since all these four finite linear orders are disjoint, we can place $\alpha_{i+1}$ before $\alpha_{i}$
  and $\beta_{i+1}$ before $\beta_{i}$ to obtain
  \[\alpha_{i+1} + \alpha_{i} + \beta_{i+1} + \beta_{i} \Vdash_\Gamma x_i < y_i < x_{i+1} < y_{i+1}.\]
\end{proof}

Analogous to the relation $\subseteq^\star$ between sets, for two infinite sequences of elements $\overline{x}$ and $\overline{y}$,
let us denote by $\overline{x} <^\star \overline{y}$ the following sentence
\[\bigvee_{q\in\omega}\bigwedge_{i,j>q}x_i < y_j.\]
Similarly, if $\A$, $\B$ are linear orders, we will slightly abuse the notation and write $\A +^\star \B$
for a suitable structure obtained from placing the elements of $\B$ just after the elements of $\A$,
with the exception of finitely many elements of $\A$ and $\B$, which may be mixed together.

\begin{proposition}\label{prop:no-chains}
  Suppose $\Gamma : \{\omega \cdot 2, \omega^\star \cdot 2\} \leq_c \{\mathcal{C},\mathcal{D}\}$, where
  $\mathcal{C}$ is a linear order without infinite descending chains and $\mathcal{D}$ is an infinite order without infinite ascending chains.
  For any two sequences of $\Gamma$-pairs $(x_i,\alpha_i)_{i\in\omega}$ and $(y_i,\beta_i)_{i\in\omega}$,
  \begin{equation}
    \label{eq:5}
    \Gamma(\sum_{i\in\omega} \alpha_i + \sum_{i\in\omega} \beta_i) \models \overline{x} <^\star \overline{y}\ \lor\ \overline{y} <^\star \overline{x},
  \end{equation}
  \begin{equation}
    \label{eq:6}
    \Gamma(\sum_{i\in\omega^\star} \alpha_i + \sum_{i\in\omega^\star} \beta_i) \models \overline{x} <^\star \overline{y}\ \lor\ \overline{y} <^\star \overline{x},
  \end{equation}  
\end{proposition}
\begin{proof}
  For (\ref{eq:5}), assume that there exist two sequences of $\Gamma$-pairs $(x_i,\alpha_{i})_{i\in\omega}$ and $(y_i,\beta_{i})_{i\in\omega}$ that witness the opposite.
  We will show that we can build two infinite subsequences $(x_{s_i},\alpha_{s_i})_{i\in\omega}$ and $(y_{t_i},\beta_{t_i})_{i\in\omega}$, such that
  \[\Gamma(\sum_{i\in\omega} \alpha_{s_i} + \sum_{i\in\omega} \beta_{t_i}) \models \bigwedge_{i} x_{s_i} < y_{t_i} < x_{s_{i+1}}.\]
  Then we will apply Proposition~\ref{prop:ascending:chain} to reach a contradiction.
  Suppose we have built the subsequences up to index $\ell$, i.e. we have the finite subsequences of $\Gamma$-pairs
  $(x_{s_i},\alpha_{s_i})_{i\leq\ell}$ and $(y_{t_i},\beta_{t_i})_{i<\ell}$, such that
  \[\alpha_{s_{0}} + \alpha_{s_1} + \cdots + \alpha_{s_{\ell}} + \beta_{t_{0}} + \beta_{t_1} + \cdots + \beta_{t_{\ell-1}} \Vdash_\Gamma \bigwedge_{i<\ell} x_{s_i} < y_{t_i} < x_{s_{i+1}}.\]
  Given indices $s_\ell$ and $t_{\ell-1}$, we start with some indices $i$ and $j$ such that $s_\ell < i$, $t_{\ell-1} < j$, and
  \[\Gamma(\sum_{k\in\omega} \alpha_k + \sum_{k\in\omega} \beta_k) \models x_{s_\ell} < x_{i} < y_{j}.\]
  Now we find some indices $j'$ and $i'$ such that $i < i'$, $j < j'$, and
  \[\Gamma(\sum_{i\in\omega} \alpha_i + \sum_{i\in\omega} \beta_i) \models y_{j} < y_{j'} < x_{i'}.\]
  Since $\Gamma(\sum_{i\in\omega} \alpha_i + \sum_{i\in\omega} \beta_i)$ does not contain an infinite descending chain, and by our assumption, we know that we can find such indices.
  We let $s_{\ell+1} = i'$ and $t_{\ell} = j'$. By the properties of $\Vdash_\Gamma$, it is clear that
  \[\alpha_{s_{0}} + \alpha_{s_1} + \cdots + \alpha_{s_{\ell+1}} + \beta_{t_{0}} + \beta_{t_1} + \cdots + \beta_{t_\ell} \Vdash_\Gamma \bigwedge_{i<\ell+1} x_{s_i} < y_{t_i} < x_{s_{i+1}}.\]
  Now by Proposition~\ref{prop:ascending:chain}, we have the following:
  \[\Gamma(\sum_{i\in\omega^\star} \alpha_i + \sum_{i\in\omega^\star} \beta_i ) \models \bigwedge_{i\in\omega} x_{s_i} < y_{t_i} < x_{s_{i+1}},\]
  which is a contradiction with the fact that $\Gamma(\sum_{i\in\omega^\star} \alpha_i + \sum_{i\in\omega^\star} \beta_i)$ does not contain an infinite ascending chain.

  The proof of (\ref{eq:6}) is symmetrical to that of (\ref{eq:5}).
\end{proof}

\begin{theorem}\label{th:include-omega-k}
  Fix some $k \geq 2$ and suppose that $\Gamma : \{\omega \cdot k, \omega^\star \cdot k\} \leq_c \{\mathcal{D}_0,\mathcal{D}_1\}$, where
  $\mathcal{D}_0$ is a linear order without infinite descending chains and $\mathcal{D}_1$ is an infinite order without infinite ascending chains.
  Then $\mathcal{D}_0$ contains $\omega \cdot k$ as a substructure, and $\mathcal{D}_1$ contains $\omega^\star \cdot k$ as a substructure.
\end{theorem}
\begin{proof}
  The case of $k = 2$ is a direct corollary of Proposition~\ref{prop:no-chains}.
  By (\ref{eq:5}), there are some infinite sequences of elements $\overline{x}$ and $\overline{y}$ such that 
  $\mathcal{D}_0 \models \overline{x} <^\star \overline{y}$. Since $\mathcal{D}_0$ does not contain an infinite descending chain, it follows that $\mathcal{D}_0$
  contains $\omega \cdot 2$ as a substructure. Similarly, property (\ref{eq:6}) tells us that $\mathcal{D}_1$ contains $\omega^\star \cdot 2$ as a substructure.
  
  Let us consider the case of $k = 3$, the general case being a straightforward generalization.
  Let $\mathcal{A}$, $\mathcal{B}$, and $\mathcal{C}$ be copies of $\omega$, with disjoint domains, partitioned in the following way:
  \[\mathcal{A} = \sum_{i\in\omega} \alpha_i,\ \mathcal{B} = \sum_{i\in\omega} \beta_i,\text{ and } \mathcal{C} = \sum_{i\in\omega} \gamma_i.\]

  We use Proposition~\ref{prop:no-chains} at most three times. Start with two arbitrary sequences of $\Gamma$-pairs $(x_i,\alpha_{k_i})_{i<\omega}$ and $(y_i,\beta_{m_i})_{i<\omega}$ and,
  without loss of generality, suppose that there is a number $\ell_1$ such that
  \[\Gamma(\sum_{i\in\omega} \alpha_{k_i} + \sum_{i\in\omega} \beta_{m_i}) \models \bigwedge_{i,j>\ell_1} x_i < y_j.\]
  Now we take a sequence of $\Gamma$-pairs $(z_i,\gamma_{n_i})_{i<\omega}$ and we may suppose that there is a number $\ell_2$ such that
  \[\Gamma(\sum_{i\in\omega} \alpha_{k_i} + \sum_{i\in\omega} \gamma_{n_i}) \models \bigwedge_{i,j>\ell_2} x_i < z_j.\]
  We must apply Proposition~\ref{prop:no-chains} one more time the two sequences of $\Gamma$-pairs $(y_i,\beta_{k_i})_{i<\omega}$ and $(z_i,\gamma_{n_i})_{i<\omega}$, We may suppose that there is a number $\ell_3$ such that
  \[\Gamma(\sum_{i\in\omega} \beta_{m_i} + \sum_{i\in\omega} \gamma_{n_i}) \models \bigwedge_{i,j>\ell_3} y_i < z_j.\]
  By monotonicity of $\Gamma$, it follows that for $\ell_0 = \max\{\ell_1,\ell_2,\ell_3\}$, we have
  \[\Gamma(\sum_{i\in\omega}\alpha_{k_i} + \sum_{i\in\omega} \beta_{m_i} + \sum_{i\in\omega} \gamma_{n_i}) \models \bigwedge_{i,j,k > \ell_0} x_i < y_j < z_k. \]
  Again by monotonicity of $\Gamma$, we have
  $\Gamma(\mathcal{A} + \mathcal{B} + \mathcal{C}) \models \bigwedge_{i,j,k > \ell_0} x_i < y_j < z_k$.
  We conclude that $\Gamma(\mathcal{A} + \mathcal{B} + \mathcal{C})$ contains a copy of $\omega \cdot 3$ as a substructure.

  For $k > 3$, the proof is similar to that of the case $k = 3$.
\end{proof}

\begin{corollary}
  For any $k < \omega$, $\{\omega \cdot 2^{k}, \omega^\star \cdot 2^{k}\} <_c \{\omega \cdot 2^{k+1}, \omega^\star \cdot 2^{k+1}\}$.
\end{corollary}
It follows that we have the following picture:
\[\{\omega,\omega^\star\} <_c \{\omega\cdot 2, \omega^\star \cdot 2\} <_c \cdots <_c \{\omega\cdot 2^k,\omega^\star \cdot 2^k\} <_c  \cdots <_c \{1+\eta, \eta + 1\}.\]

Now we are ready to prove the main result of this paper, which is the following theorem.

\begin{theorem}
  For any two non-zero natural numbers $k$ and $t$,
  \[k\ \mid\ t\ \iff\ \{\omega \cdot k, \omega^\star \cdot k\} \leq_c \{\omega\cdot t,\omega^\star \cdot t\}.\]
\end{theorem}

First we will study the simpler case when $k = 2$ and $t = 3$ and then we will prove the general case.

\section{The case $\{\omega \cdot 2, \omega^\star \cdot 2\} \not\leq_c \{\omega\cdot 3,\omega^\star \cdot 3\}$}
In this section, towards a contradiction, assume $\{\omega \cdot 2, \omega^\star \cdot 2\} \leq_c \{\omega\cdot 3,\omega^\star \cdot 3\}$ via the enumeration operator $\Gamma$.

\begin{claim}\label{claim:no-omega-2}
  There is no copy $\A$ of $\omega$ such that $\Gamma(\A)$ contains a copy of $\omega \cdot 2$ as a substructure.
\end{claim}
\begin{proof}
  Assume $\A$ is a copy of $\omega$ such that $\Gamma(\A)$ contains a copy of $\omega \cdot 2$.
  By Proposition~\ref{prop:alpha:finite}, it follows that we can take a sequence of finite linear orders $\{\alpha_i\}_{i\in\omega}$
  such that $\sum_{i\in\omega} \alpha_i \subseteq \mathcal{A}$,
  $\Gamma(\sum_{i\in\omega} \alpha_i) \models \bigwedge_{i,j \in \omega} x_i < y_j$, and
  $x_i,y_i \in \Gamma(\alpha_i)$.
  But then we can form the linear order $\B = \sum_{i\in\omega} \alpha_{2i} + \sum_{i\in\omega} \alpha_{2i+1}$ of type $\omega \cdot 2$ and by
  Proposition~\ref{prop:no-chains},
  $\Gamma(\B)$ will contain a copy of $\omega \cdot 4$, which is a contradiction.
\end{proof}

\begin{claim}\label{claim:omega-to-omega}
  If $\mathcal{A}$ is a copy of $\omega$, then $\Gamma(\mathcal{A})$ is a copy of $\omega$.
\end{claim}
\begin{proof}
  Fix a copy $\A$ of $\omega$ and assume that $\Gamma(\A) \cong \omega + \ell$, where $\ell > 0$.
  Let $y$ be the greatest element in $\Gamma(\A)$ and let $\hat\alpha$ be an initial segment of $\A$
  such that $y \in \Gamma(\hat\alpha)$.
  Let $\B$ be another copy of $\omega$.
  Since $\Gamma(\A+\B) \cong \omega \cdot 3$,
  we can partition $\A\setminus\hat\alpha$ and $\B$ into finite suborderings $\alpha_i$ and $\beta_i$, respectively, and choose elements $x_i$ and $z_i$ such that for all $i$, $\hat\alpha + \alpha_i + \beta_i \Vdash_\Gamma x_i < y < z_i$.
  It follows that $\Gamma(\hat\alpha + \sum_{i\in\omega}(\alpha_i+\beta_i))$ contains a copy of $\omega \cdot 2$, which is a contradiction by Claim~\ref{claim:no-omega-2}.
\end{proof}

\begin{claim}\label{claim:no-chains:2-3}
  For any two sequences of $\Gamma$-pairs $(x_i,\alpha_i)_{i\in\omega}$ and $(y_i,\beta_i)_{i\in\omega}$,
  \[\Gamma(\sum_{i\in\omega} \alpha_i + \sum_{i\in\omega} \beta_i) \models \overline{x} <^\star \overline{y}.\]
\end{claim}
\begin{proof}
  In view of Proposition~\ref{prop:no-chains}, towards a contradiction, assume that there is some $q$ such that
  \[\Gamma(\sum_{i\in\omega} \alpha_i + \sum_{i\in\omega} \beta_i) \models \bigwedge_{i,j>q} y_j < x_i.\]
  Consider some $x_i$, where $q < i$, and consider the linear order $\mathcal{A} = \alpha_i + \sum_{j \geq i} \beta_j$ of order type $\omega$.
  By Claim~\ref{claim:omega-to-omega}, $\Gamma(\mathcal{A})$ has order type $\omega$
  and hence only finitely many elements $y_j$ are to the left of $x_i$.
  But $\mathcal{A} \subseteq \sum_{i\in\omega} \alpha_i + \sum_{i\in\omega} \beta_i$.
  By monotonicity of $\Gamma$,
  only finitely many $y_j$ are to the left of $x_i$ in the linear order $\Gamma(\sum_{i\in\omega} \alpha_i + \sum_{i\in\omega} \beta_i)$.
  We reach a contradiction. Thus, by Proposition~\ref{prop:no-chains},
  \[\Gamma(\sum_{i\in\omega} \alpha_i + \sum_{i\in\omega} \beta_i) \models \overline{x} <^\star \overline{y}.\]
\end{proof}

Let $\mathcal{A}$, $\mathcal{B}$ and $\mathcal{C}$ be linear orders of type $\omega$.
According to Claim~\ref{claim:omega-to-omega} we have that $\Gamma(\mathcal{A})$ and $\Gamma(\mathcal{B})$ have order types $\omega$
and by Claim~\ref{claim:no-chains:2-3} we have one of the following three possible cases:
\begin{align}
  \Gamma(\mathcal{A}+\mathcal{B}) & = \Gamma(\mathcal{A}) +^\star \Gamma(\mathcal{B}) +^\star \mathcal{C} \label{case:right}\\
  \Gamma(\mathcal{A}+\mathcal{B}) & = \mathcal{C} +^\star \Gamma(\mathcal{A}) +^\star \Gamma(\mathcal{B}) \label{case:left}\\
  \Gamma(\mathcal{A}+\mathcal{B}) & = \Gamma(\mathcal{A}) +^\star \mathcal{C} +^\star \Gamma(\mathcal{B}) \label{case:middle}.
\end{align}

In view of Proposition~\ref{prop:alpha:finite}, we may suppose that the linear orders are partitioned
so that $\mathcal{A} = \sum_{i\in\omega} \alpha_i$,
$\mathcal{B} = \sum_{i\in\omega} \beta_i$ and we have the distinct elements
$x_i \in \Gamma(\alpha_i)$, $y_i \in \Gamma(\beta_i)$ and $z_i \in \Gamma(\alpha_i+\beta_i)$,
which belong to the domains of $\Gamma(\mathcal{A})$, $\Gamma(\mathcal{B})$ and $\mathcal{C}$, respectively.
We will show that none of the above three cases are possible.

For Case~(\ref{case:right}), we may assume that
\[\Gamma(\sum_{i\in\omega} \alpha_i + \sum_{i\in\omega} \beta_i) \models \bigwedge_{i,j,k \in \omega} x_i < y_j < z_k.\]
Then $\alpha_0 + \beta_0 \Vdash_\Gamma y_0 < z_0$ and for any $i > 0$, $\alpha_0 + \beta_0 + \beta_i \Vdash_\Gamma y_i < z_0$.
It follows that
\[\Gamma(\alpha_0 + \sum_{i\in\omega}\beta_i) \models \bigwedge_{i\in\omega}y_i < z_0.\]
Hence $\Gamma(\alpha_0 + \sum_{i\in\omega}\beta_i)$ contains a copy of $\omega + 1$, which is a contradiction by Claim~\ref{claim:omega-to-omega}.

For Case~(\ref{case:left}), we may assume that
\[\Gamma(\sum_{i\in\omega} \alpha_i + \sum_{i\in\omega} \beta_i) \models \bigwedge_{i,j,k \in \omega} z_k < x_i < y_j.\]
Then $\alpha_0 + \beta_0 \Vdash_\Gamma z_0 < x_0$ and for any $i > 0$, $\alpha_0 + \alpha_i + \beta_i \Vdash_\Gamma z_i < x_0$.
It follows that
\[\Gamma(\sum_{i\in\omega}(\alpha_i+\beta_i)) \models \bigwedge_{i\in\omega} z_i < x_0.\]

Hence $\Gamma(\sum_{i\in\omega}(\alpha_i + \beta_i))$ contains a copy of $\omega + 1$, which is a contradiction by Claim~\ref{claim:omega-to-omega}.

Now for Case~(\ref{case:middle}), we may assume that
\begin{equation}
  \label{eq:2}
  \Gamma(\sum_{i\in\omega} \alpha_i + \sum_{i\in\omega} \beta_i) \models \bigwedge_{i,j,k \in \omega} x_i < z_k < y_j.
\end{equation}
We know that for all $i$, $\alpha_i + \beta_0 \Vdash_\Gamma x_i < y_0$.
We also know that
\[\Gamma(\sum_{i\in\omega}(\alpha_i + \beta_i)) \not\models \overline{x} <^\star \overline{y}.\]
Otherwise, we would get a contradiction with Claim~\ref{claim:no-omega-2}.
It follows that there is some index $\ell>0$ such that
\[\Gamma(\sum_{i\in\omega}(\alpha_i + \beta_i)) \models \bigwedge_{j>\ell}y_0 < x_j.\]
Hence $\beta_0 + \alpha_{\ell+1} \Vdash_\Gamma y_0 < x_{\ell+1}$.
Moreover, (\ref{eq:2}) implies
\[\alpha_{\ell} + \beta_0 + \beta_{\ell} \Vdash_\Gamma z_{\ell} < y_0\text{ and } \alpha_{\ell} + \alpha_{\ell+1} + \beta_{\ell} \Vdash_\Gamma x_{\ell+1} < z_{\ell}.\]
Combining all of the above, by monotonicity of enumeration operators, it follows that
\[\alpha_{\ell} + \beta_0 + \alpha_{\ell+1} + \beta_{\ell} \Vdash_\Gamma y_0 < x_{\ell+1} < z_{\ell} < y_0,\]
which is a contradiction.  

We considered all possible cases and in each one of them we reach a contradiction. Thus, we conclude that
$\{\omega \cdot 2, \omega^\star \cdot 2\} \not\leq_c \{\omega \cdot 3, \omega^\star \cdot 3\}$.

\section{The general case}

Suppose that $\{\omega\cdot k, \omega^\star \cdot k\} \leq_c \{\omega\cdot t, \omega^\star \cdot t\}$ via the enumeration operator $\Gamma$, where $t = km + r$ for some $m$ and $0 \leq r < k$.
By Theorem~\ref{th:include-omega-k} we know that $m \geq 1$.
It is straightforward to see that for each $m \geq 1$ there is an enumeration operator $\Gamma_m$ such that $\Gamma_m:\{\omega\cdot k, \omega^\star \cdot k\} \leq_c \{\omega\cdot (km), \omega^\star \cdot (km)\}$.
This $\Gamma_m$ just copies the input structure $m$ number of times. In this section we will show that the enumeration operators cannot be any ``smarter'' than this, i.e. it is only possible to have $r = 0$.

\begin{proposition}\label{prop:no-omega-tail}
  If $\A$ is a copy of $\omega$ such that $\Gamma(\A) \cong \omega \cdot m + \ell$, where $\ell > 0$,
  then there is another copy $\hat\A$ of $\omega$ with $dom(\hat\A) = dom(\A)$ such that $\Gamma(\hat\A)$ has the type of a limit ordinal at least $\omega \cdot (m+1)$.
\end{proposition}
\begin{proof}
  By Proposition~\ref{prop:alpha:finite}, we can choose a sequence $\alpha_i$ of finite suborderings of $\A$
  such that for $\B = \sum_i \alpha_i$, $\Gamma(\B) \cong \Gamma(\A)$ and the set
  $C = \text{dom}(\A) \setminus \text{dom}(\B)$ is infinite.
  We can order the elements of $C$ in a linear order $\mathcal{C}$ of type $\omega \cdot (k-1)$.
  Let $a_0$ be the greatest element in $\Gamma(\B)$.
  Since $\B + \C$ has type $\omega \cdot k$ and $\Gamma(\B+\C)$ has no greatest element, there is some $c_0$
  such that $\Gamma(\B+\C) \models a_0 < c_0$.
  Fix some finite initial segments $\beta_0$ of $\B$ and $\gamma_0$ of $\C$
  such that $\beta_0 + \gamma_0 \Vdash_\Gamma a_0 < c_0$.
  Let $\B_0 = \B \setminus \beta_0$ and $\C_0 = \C \setminus \gamma_0$.
  Consider the linear order $\hat\B_0 = \beta_0 + \gamma_0 + \B_0$ of type $\omega$.
  Clearly, $\Gamma(\hat\B_0) \models a_0 < c_0$ and since $\B$ is included in $\hat\B_0$,
  $\Gamma(\hat\B_0)$ has a type at least $\omega \cdot m + \ell + 1$.

  Now we let $a_1$ be the greatest element in $\Gamma(\hat\B_0)$, if such exists, or
  we let $a_1 = c_0$.
  Since $\hat\B_0 + \C_0$ has type $\omega \cdot k$ and $\Gamma(\hat\B_0 + \C_0)$ has no greatest element,
  there is some $c_1$ such that $\Gamma(\hat\B_0 + \C_0) \models a_0 < c_0 \leq a_1 < c_1$.
  Fix some finite initial segments $\beta_1$ of $\B_0$ and $\gamma_1$ of $\C_0$
  such that $\beta_0 + \gamma_0 + \beta_1 + \gamma_1 \Vdash_\Gamma a_0 < c_1 \leq a_1 < c_1$.
  Let $\B_1 = \B_0 \setminus \beta_1$ and $\C_1 = \C_0 \setminus \gamma_1$.
  Consider the linear order $\hat\B_1 = \beta_0 + \gamma_0 + \beta_1 + \gamma_1 + \B_1$ of type $\omega$.
  Clearly, $\Gamma(\hat\B_1) \models a_0 < c_0 \leq a_1 < c_1$ and since $\hat\B_0$ is included in $\hat\B_1$,
  $\Gamma(\hat\B_1)$ has a type at least $\omega \cdot m + \ell + 2$.

  It is clear that we can iterate this construction for every $i$ and produce
  $\hat\B_i = \sum_{j \leq i} (\beta_j + \gamma_j) + \mathcal{B}_{i}$ of type $\omega$
  such that $\Gamma(\hat\B_i)$ has order type at least $\omega\cdot m + \ell + i + 1$.
  In the end we let
  \[\hat\A = \sum_{i\in\omega} (\beta_i + \gamma_i)\] 
  such that $\Gamma(\hat\A)$ has order type at least $\omega \cdot (m+1)$.
  To finish the proof we must show that $\Gamma(\hat\A)$ has the type of a limit ordinal.
  Assume that $\Gamma(\hat\A)$ has a greatest element $d$.
  Fix some index $q$ such that $d \in \Gamma(\sum_{i \leq q}(\beta_i + \gamma_i))$.
  But then $d \in \Gamma(\hat\B_q)$ and since $\hat\B_q \subseteq \hat\A$, by monotonicity, it follows that $d$ is the greatest
  element in $\Gamma(\hat\B_q)$, i.e. $d = a_{q+1}$. At the next stage we will find an element $c_{q+1}$ such that
  $\Gamma(\hat\B_{q+1}) \models d < c_{q+1}$ and by monotonicity, $\Gamma(\hat\A) \models d < c_{q+1}$.
\end{proof}

For an infinite set $M$, let us denote
\[\text{Ord}(M) = \sup\{\Gamma(\mathcal{A}) \mid \mathcal{A} \cong \omega\ \&\ \text{dom}(\mathcal{A}) \subseteq M\}.\]

By Proposition~\ref{prop:no-omega-tail},
$\text{Ord}(M)$ is a limit ordinal.
Let us consider sequences of sets $M_0,M_1,\dots$, where $M_0 = \mathbb{N}$ and $\text{Ord}(M_{i+1}) < \text{Ord}(M_i)$.
We know that any such sequence is finite.
Fix one such finite sequence ${(M_i)}_{i\leq p}$, which cannot be extended further.
Let $m$ be such that $\omega \cdot m = \text{Ord}(M_p)$.
Any copy $\A$ with $\text{dom}(\A) \subseteq M_p$ is such that $\Gamma(\A) \cong \omega \cdot m$.
Otherwise, we would be able to extend the finite sequence with another set.
From now on in this section, we will always suppose that we work with copies of $\omega$ whose domains are subsets of $M_p$
and thus any copy of $\omega$ will produce a copy of $\omega \cdot m$ via $\Gamma$.

The next proposition is a generalization of Claim~\ref{claim:no-chains:2-3}.

\begin{proposition}\label{prop:no-merge-general}
  Let $\A$, $\A_1,\dots,\A_m$ and $\B$, $\B_1,\dots,\B_m$ be copies of $\omega$ such that
  $\Gamma(\A) = \A_1 + \A_2 + \cdots + \A_m$ and
  $\Gamma(\B) = \B_1 + \B_2 + \cdots + \B_m$.
  Then
  \[\Gamma(\A+\B) \supseteq \A_1 +^\star \B_1 + \A_2 +^\star \B_2 + \cdots + \A_m +^\star \B_m.\]
\end{proposition}
\begin{proof}
  Let us first partition $\A$ and $\B$ into finite suborderings $\alpha_i$ and $\beta_i$ such that
  $\A = \sum_{i\in\omega}\alpha_i$, $\B = \sum_{i\in\omega} \beta_i$ and for all $i \in \omega$,
  $x^i_1,\dots,x^i_m \in \Gamma(\alpha_i)$ and $y^i_1,\dots,y^i_m \in \Gamma(\beta_i)$, where
  $\overline{x}_j = (x^i_j)_{i\in\omega}$ and $\overline{y}_j = (y^i_j)_{i\in\omega}$, for $1 \leq j \leq m$,
  represent $\omega$-chains of distinct elements belonging to $\A_j$ and $\B_j$, respectively.
  Let $\C = \sum_{i\in\omega}(\alpha_i + \beta_i)$.
  Since $\Gamma(\C) \cong \omega \cdot m$, and $\Gamma(\C) \models \overline{x}_1 < \overline{x}_2 < \cdots < \overline{x}_m$
  and $\Gamma(\C) \models \overline{y}_1 < \overline{y}_2 < \cdots < \overline{y}_m$, we have
  $\Gamma(\C) \models \overline{x}_j < \overline{y}_{j+1}$ and $\Gamma(\C) \models \overline{y}_j < \overline{x}_{j+1}$
  for $1 \leq j < m$.
  It follows that for all $i$,
  \[\alpha_i + \beta_i \Vdash_\Gamma \bigwedge^{m-1}_{j=1} x^i_j < y^i_{j+1}\ \&\ y^i_j < x^i_{j+1}.\]
  Hence
  \[\Gamma(\A+\B) \models \bigwedge^{m-1}_{j=1}\overline{x}_j < \overline{y}_{j+1}\ \&\ \overline{y}_j < \overline{x}_{j+1}.\]
  We can adapt the proof of Proposition~\ref{prop:no-chains} to show that for $1 \leq j \leq m$,
  \[\Gamma(\A+\B) \models \overline{x}_j <^\star \overline{y}_j\ \lor\ \overline{y}_j <^\star \overline{x}_j.\]
  But since $\Gamma(\C) \cong \omega \cdot m$, the chains $\overline{x}_j$ and $\overline{y}_j$ must be merged in $\Gamma(\C)$.
  Otherwise, $\Gamma(\C)$ would contain a copy of $\omega \cdot (m+1)$.
  It follows that there is some $q$ such that for all $\ell > q$,
  $\Gamma(\C) \models \bigwedge^m_{j=1}x^0_j < y^\ell_j$.
  Since $y^\ell_1,\dots,y^\ell_m \in \Gamma(\beta_\ell)$ and $x^0_1,\dots,x^0_m \in \Gamma(\alpha_0)$, we have
  \[\alpha_0 + \beta_\ell \Vdash_\Gamma \bigwedge^m_{j=1}x^0_j < y^\ell_j,\]
  and hence $\Gamma(\A+\B) \models \bigwedge^m_{j=1}\overline{x}_j <^\star \overline{y}_j$.
  We conclude that
  \[\Gamma(\A + \B) \models \overline{x}_1 <^\star \overline{y}_1 < \overline{x}_2 <^\star \overline{y}_2 < \cdots < \overline{x}_m <^\star \overline{y}_m.\]
\end{proof}


\begin{corollary}\label{cor:no-merge-general}
  Let $\A_1,\dots,\A_k$ be copies of $\omega$ such that
  $\Gamma(\A_i) = \B^1_i + \cdots + \B^m_i$,
  where $\B^j_i$ are also copies of $\omega$.
  Then
  \[\Gamma(\sum^k_{i=1} \A_i) \supseteq \sum^m_{j=1}(\B^j_1 +^\star \B^j_2 +^\star \cdots +^\star \B^j_k).\]
\end{corollary}
\begin{proof}
  We essentially use the monotonicity property of enumeration operators.
  For simplicity, consider the case of $k = 3$.
  By Proposition~\ref{prop:no-merge-general}, we know that
  \begin{align*}
    \Gamma(\A_1 + \A_2) & \supseteq \B^1_1 +^\star \B^1_2 + \cdots + \B^m_1 +^\star \B^m_2,\\
    \Gamma(\A_2 + \A_3) & \supseteq \B^1_2 +^\star \B^1_3 + \cdots + \B^m_2 +^\star \B^m_3,\\
    \Gamma(\A_1 + \A_3) & \supseteq \B^1_1 +^\star \B^1_3 + \cdots + \B^m_1 +^\star \B^m_3.
  \end{align*}
  Since $\Gamma(\A_1+\A_2)$, $\Gamma(\A_2+\A_3)$ and $\Gamma(\A_1+\A_3)$ are included in $\Gamma(\A_1+\A_2+\A_3)$, we obtain
  for $1 \leq i \leq m$ that $\Gamma(\A_1+\A_2+\A_3) \supseteq \B^i_1 +^\star \B^i_2 +^\star \B^i_3$,
  and for $1 \leq i < m$, $\Gamma(\A_1+\A_2+\A_3) \supseteq \B^i_3 + \B^{i+1}_1$.
  We conclude that 
  \[\Gamma(\A_1 + \A_2 + \A_3) \supseteq \B^1_1 +^\star \B^1_2 +^\star \B^1_3 + \cdots + \B^m_1 +^\star \B^m_2 +^\star \B^m_3.\]
\end{proof}

For $i=1,\dots,k$, let $\Gamma(\A_i) = \B^1_i + \cdots + \B^m_i$.
Recall that our enumeration operator $\Gamma$ is such that $\Gamma:\{\omega\cdot k, \omega^\star \cdot k\} \leq_c \{\omega\cdot t, \omega^\star \cdot t\}$, where $t = kq + r$ for some $q \geq 1$ and $0 \leq r < k$. 

If we assume that $r > 0$, by Corollary~\ref{cor:no-merge-general}, we have the following cases to consider
where an extra copy of $\omega$, denoted by $\C$, is placed:
\begin{align}
  \Gamma(\sum^k_{i=1} \A_i)   & \supseteq \B^m_k +^\star \C; \label{eq:right-general}\\
  \Gamma(\sum^k_{i = 1} \A_i) & \supseteq \C +^\star \B^1_1; \label{eq:left-general}\\
  \Gamma(\sum^k_{i=1} \A_i)  & \supseteq \B^j_k +^\star \C +^\star \B^{j+1}_1, \text{ for some }j < m; \label{eq:3}\\
  \Gamma(\sum^k_{i=1} \A_i)  & \supseteq \B^j_\ell +^\star \C +^\star \B^j_{\ell+1}\text{, for some }j \leq m\text{ and }\ell < k. \label{eq:4}
\end{align}
We will show that none of these cases are possible and conclude that $r = 0$.

For Case~(\ref{eq:right-general}), we proceed as in Case~(\ref{case:right}). We may assume
\[\Gamma(\sum^k_{i=1} \A_i) \models \bigwedge_{\ell,k\in\omega} x_\ell < z_k,\]
where $\B^m_k$ and $\C$ are represented by the $\omega$-chains $(x_\ell)_{\ell\in\omega}$ and $(z_\ell)_{\ell\in\omega}$, respectively.
Moreover, we may assume that $\A_i = \sum_{\ell\in\omega} \alpha^\ell_i$ for $1\leq i \leq k$,
and for all $\ell$, $x_\ell \in \Gamma(\alpha^\ell_k)$ and $z_\ell \in \Gamma(\sum^k_{i=1} \alpha^\ell_i)$.
Then we have the following:
\[\alpha^0_1 + \cdots + \alpha^0_{k-1} + \alpha^0_k \Vdash_\Gamma x_0 < z_0,\]
and for the linear order $\D = \alpha^0_1 + \cdots + \alpha^0_{k-1} + \A_k$ of type $\omega$,
\[\Gamma(\D) \models \bigwedge_{\ell\in\omega} x_\ell < z_0.\]
Since $\A_k$ is included in $\D$, it follows that $\Gamma(\A_k)$ is included in $\Gamma(\D)$ and hence
$\Gamma(\D)$ contains a linear order of type $\omega \cdot m + 1$.
We reach a contradiction.

For Case~(\ref{eq:left-general}), we proceed as in Case~(\ref{case:left}). We may assume
\[\Gamma(\sum^k_{i=1} \A_i) \models \bigwedge_{\ell,k\in\omega} z_k < x_\ell,\]
where $\B^1_1$ and $\C$ are represented by the $\omega$-chains ${(x_\ell)}_{\ell\in\omega}$ and ${(z_\ell)}_{\ell\in\omega}$, respectively.
Moreover, we may assume that $\A_i = \sum_{\ell\in\omega} \alpha^\ell_i$ for $1\leq i \leq k$,
and for all $\ell$, $x_\ell \in \Gamma(\alpha^\ell_1)$ and $z_\ell \in \Gamma(\sum^k_{i=1} \alpha^\ell_i)$.
Then we have the following:
\[\sum^k_{i=1}\alpha^0_i \Vdash_\Gamma z_0 < x_0,\]
and for every $\ell > 0$,
\[\alpha^0_1 + \sum^k_{i=1} \alpha^\ell_i \Vdash_\Gamma z_\ell < x_0.\]
It follows that for the linear order $\D = \sum_{\ell\in\omega}\sum^k_{j=1} \alpha^\ell_j$ of type $\omega$,
\[\Gamma(\D) \models \bigwedge_{t\in\omega} z_t < x_0.\]
Since $\A_1$ is included in $\D$, it follows that
$\Gamma(\D)$ contains a linear order of type $\omega \cdot (m+1)$. We reach a contradiction.

For Case~(\ref{eq:3}), we may assume
\begin{equation}
  \label{eq:7}
  \Gamma(\sum^k_{i=1} \A_i) \models \bigwedge_{s,t,v\in\omega} x_s < z_t < y_v,
\end{equation}
where $\B^j_k$, $\C$ and $\B^{j+1}_{1}$ are represented by the $\omega$-chains ${(x_\ell)}_{\ell\in\omega}$, ${(z_\ell)}_{\ell\in\omega}$,
and ${(y_\ell)}_{\ell\in\omega}$, respectively.
Moreover, we may assume that $\A_i = \sum_{\ell\in\omega} \alpha^\ell_i$ for $1\leq i \leq k$,
and for all $\ell$, $x_\ell \in \Gamma(\alpha^\ell_k)$, $y_\ell \in \Gamma(\alpha^\ell_1)$,
and $z_\ell \in \Gamma(\sum^k_{i=1} \alpha^\ell_i)$.
We have that for all $\ell$,
\[\sum^k_{i=1} \alpha^\ell_i \Vdash_\Gamma x_\ell < z_\ell < y_\ell.\]
Consider some indices $s$, $t$, and $v$.
As we already noted, if $s = t = v$, then $\sum^k_{i=1} \alpha^t_i \Vdash_\Gamma x_t < z_t < y_t$.
If $t < s$, then
\[\sum^{k}_{i=1} \alpha^t_i + \alpha^s_k \Vdash_\Gamma x_s < z_t,\]
and if $t > s$, then since $\alpha^s_k + \alpha^t_k \Vdash_\Gamma x_s < x_t$,
we obtain
\[\alpha^s_k + \sum^{k}_{i=1} \alpha^t_i \Vdash_\Gamma x_s < x_t < z_t.\]
Similarly, if $t > v$, then
\[\alpha^v_1 + \sum^{k}_{i=1} \alpha^t_i \Vdash_\Gamma z_t < y_v,\]
and if $t < v$, then since $\alpha^t_1 + \alpha^v_1 \Vdash_\Gamma y_t < y_v$,
we obtain
\[\sum^{k}_{i=1} \alpha^t_i + \alpha^v_1 \Vdash_\Gamma z_t < y_t < y_v.\]
We conclude that for the linear order $\D = \sum_{\ell\in\omega}\sum^k_{j=1} \alpha^\ell_j$ of type $\omega$, we have 
\[\Gamma(\D) \models \bigwedge_{s,t,v\in\omega}x_s < z_t < y_v.\]
In other words,
\[\Gamma(\D) \supseteq \B^j_k +^\star \C +^\star \B^{j+1}_{1}.\]
Since $\A_1, \A_k \subseteq \D$, we also have
\begin{align*}
  \Gamma(\D) & \supseteq \B^1_k + \B^2_k + \cdots + \B^j_k,\\
  \Gamma(\D) & \supseteq \B^{j+1}_1 + \B^{j+2}_1 + \cdots + \B^m_1.
\end{align*}
Combining all of the above, we conclude that
\[\Gamma(\D) \supseteq \B^1_k + \B^2_k + \cdots + \B^j_k +^\star \C +^\star \B^{j+1}_1 + \B^{j+2}_1 + \cdots \B^m_1,\]
which means that $\Gamma(\D)$ contains a copy of $\omega \cdot (m+1)$,
which is a contradiction.

For Case~(\ref{eq:4}), we may assume that
\[\Gamma(\sum^k_{i=1} \A_i) \models \bigwedge_{i,j,k\in\omega} x_i < z_k < y_j,\]
where $\B^j_i$, $\C$ and $\B^{j}_{i+1}$ are represented by the $\omega$-chains ${(x_\ell)}_{\ell\in\omega}$, ${(z_\ell)}_{\ell\in\omega}$,
and ${(y_\ell)}_{\ell\in\omega}$, respectively.
Moreover, we may assume that $\A_i = \sum_{\ell\in\omega} \alpha^\ell_i$ for $1\leq i \leq k$,
and $x_\ell \in \Gamma(\alpha^\ell_i)$, $y_\ell \in \Gamma(\alpha^\ell_{i+1})$,
and $z_\ell \in \Gamma(\sum^k_{i=1} \alpha^\ell_i)$ for all $\ell$.

We will reach a contradiction in a similar manner as in Case~(\ref{case:middle}).
Consider the linear order $\D = \sum_{\ell\in\omega}(\alpha^\ell_i + \alpha^\ell_{i+1})$ of type $\omega$.
Since $\A_i$ and $\A_{i+1}$ are included in $\D$,
by monotonicity, $\Gamma(\D)$ contains both $\B^1_i+\cdots+\B^m_i$ and $\B^1_{i+1} + \cdots + \B^m_{i+1}$.
Since $\Gamma(\D) \cong \omega \cdot m$,
the linear orders $\B^j_i$ and $\B^j_{i+1}$ are merged in $\Gamma(\D)$.
It follows that there is some index $\ell > 0$ such that
\[\alpha^0_{i+1} + \alpha^{\ell+1}_i  \Vdash_\Gamma y_0 < x_{\ell+1} .\]
On the other hand, since $y_0 \in \Gamma(\alpha^0_{i+1})$ and $z_\ell \in \Gamma(\sum^k_{i=1}\alpha^\ell_i)$, we obtain
\[\sum^i_{j=1} \alpha^{\ell}_j + \alpha^0_{i+1} + \sum^k_{j=i+1} \alpha^{\ell}_j \Vdash_\Gamma z_{\ell} < y_0.\]
Moreover, since $x_{\ell+1} \in \Gamma(\alpha^{\ell+1}_i)$, we obtain
\[\sum^i_{j=1} \alpha^\ell_j + \alpha^{\ell+1}_i + \sum^k_{j=i+1}\alpha^\ell_j \Vdash_\Gamma x_{\ell+1} < z_\ell,\]
Combining all of the above, we obtain the following
\[\sum^i_{j=1} \alpha^\ell_j + \alpha^{\ell+1}_i +\alpha^0_{i+1} + \sum^k_{j=i+1}\alpha^\ell_j \Vdash_\Gamma y_0 < x_{\ell+1} < z_\ell < y_0,\]
which is a contradiction.

We considered all possible cases and reached a contradiction in each one of them.
Thus, we proved our main result.

\bibliographystyle{alpha}
\bibliography{pairs}

\end{document}